\documentclass[pdflatex,sn-mathphys-ay]{sn-jnl}

\usepackage{graphicx}%
\usepackage{multirow}%
\usepackage{amssymb}
\usepackage{natbib}
\usepackage{amsmath}
\usepackage{amsfonts}
\usepackage{mathrsfs}
\usepackage{mathtools}
\usepackage[title]{appendix}%
\usepackage{xcolor}%
\usepackage{textcomp}%
\usepackage{manyfoot}%
\usepackage{booktabs}%
\usepackage{algorithm}%
\usepackage{algorithmicx}%
\usepackage{algpseudocode}%
\usepackage{listings}%
\usepackage{enumitem}

\usepackage{tikz}
\usetikzlibrary{arrows.meta,positioning,fit,calc,decorations.pathreplacing,shapes.multipart}
\usepackage{cleveref}

\theoremstyle{thmstyleone}
\newtheorem{theorem}{Theorem}[section]
\newtheorem{proposition}[theorem]{Proposition}
\newtheorem{lemma}[theorem]{Lemma}
\newtheorem{corollary}[theorem]{Corollary}
\theoremstyle{definition}
\newtheorem{definition}[theorem]{Definition}

\theoremstyle{remark}
\newtheorem{remark}[theorem]{Remark}

\begin{document}

\title{Mode-Wise Spectral Criteria for Coupled Mass Transport in Hybrid PDE–ODE Tumor Microenvironments}

\author[1]{\fnm{Jiguang} \sur{Yu}}\email{jyu678@bu.edu}
\equalcont{These authors contributed equally to this work as co-first authors.}

\author*[2,3]{\fnm{Louis Shuo} \sur{Wang}}\email{swang116@vols.utk.edu}
\equalcont{These authors contributed equally to this work as co-first authors.}

\author[4]{\fnm{Zonghao} \sur{Liu}} \email{liuzonghao@fjmu.edu.cn}

\author*[4]{\fnm{Jingfeng} \sur{Liu}}\email{drjingfeng@126.com}

\affil[1]{\orgdiv{College of Engineering},
  \orgname{Boston University},
  \orgaddress{\city{Boston}, \postcode{02215}, \state{MA}, \country{United States}}}

\affil[2]{\orgdiv{Department of Integrated Systems Engineering},
  \orgname{Ohio State University},
  \orgaddress{\city{Columbus}, \postcode{43210}, \state{OH}, \country{United States}}}

\affil[3]{\orgdiv{Department of Mathematics},
  \orgname{University of Tennessee},
  \orgaddress{\city{Knoxville}, \postcode{37996}, \state{TN}, \country{United States}}}

\affil[4]{%
  \orgdiv{Department of Hepatopancreatobiliary Surgery},
  \orgname{Fujian Cancer Hospital},
  \orgaddress{\city{Fuzhou}, \postcode{350014}, \state{Fujian},\country{China}}}

\abstract{

We study coupled mass transport in a tumor--microenvironment setting with two motile densities $(S,R)$ and non-motile state switching $(P,A)$. The populations diffuse and undergo chemotactic drift; $(P,A)$ follow pointwise ODE switching. A decoupled inhibitory field $D$ satisfies a damped Neumann heat equation, giving maximum-principle bounds and exponential decay. Together with the pointwise invariant $P+A$, these identities yield global existence, positivity, and long-time reduction to limiting $(S,R)$ kinetics with a unique globally attracting coexistence state. Neumann eigenmode reduction gives closed dispersion relations. The base $(S,R)$ reaction--diffusion block remains stable for all nonconstant modes for any $d_S,d_R>0$, excluding classical Turing destabilization. Chemotaxis is posed via a diffusive cue $c$, since $\nabla A$ is undefined for non-diffusive $A$. In one-way damped coupling, the linearized mode matrix is block triangular and leaves the $(S,R)$ spectrum unchanged. Two-way coupling adds a feedback rank-one mobility correction, induces effective cross-diffusion, and admits mode growth. We give explicit trace/determinant criteria for unstable Laplacian modes and the resulting instability thresholds.

} 

\keywords{
hybrid PDE–ODE model; reaction–diffusion–chemotaxis system; coupled transport; Turing-type instability; tumor microenvironment
}

\maketitle

\section{Introduction}
\label{sec:introduction}

Let $U\subset\mathbb{R}^N$ ($N\in\{1,2,3\}$) be a bounded domain with smooth boundary.
Consider the semilinear reaction--diffusion system
\begin{equation}
\label{eq:intro_RD_general}
\partial_t \mathbf{u}=D\Delta \mathbf{u}+F(\mathbf{u}),\qquad 
\mathbf{u}:\,U\times(0,\infty)\to\mathbb{R}^m,
\end{equation}
where $D=\mathrm{diag}(d_1,\dots,d_m)$ with $d_i> 0$ and $F:\mathbb{R}^m\to\mathbb{R}^m$ is $C^1$.
A spatially homogeneous equilibrium $\mathbf{u}^*$ satisfies $F(\mathbf{u}^*)=0$ and is linearly stable if and only if
\begin{equation}
\label{eq:intro_kinetic_stability}
\operatorname{Spec}\bigl(DF(\mathbf{u}^*)\bigr)\subset\{\Re z<0\},
\end{equation}
where $DF(\mathbf{u}^*)$ denotes the Jacobian at $\mathbf{u}^*$.
Under homogeneous Neumann boundary conditions, let $\{(\lambda_k,\omega_k)\}_{k\ge1}$ be the eigenpairs of the Neumann Laplacian,
\[
-\Delta\omega_k=\lambda_k\omega_k \ \text{in }U,\qquad 
\partial_{\mathrm n}\omega_k=0 \ \text{on }\partial U,
\qquad 0=\lambda_1<\lambda_2\le\cdots,\ \lambda_k\to\infty.
\]
Projecting the linearization at $\mathbf{u}^*$ onto a single eigenmode yields the mode matrix
\begin{equation}
\label{eq:intro_mode_matrix}
\mathcal{A}_k := DF(\mathbf{u}^*)-\lambda_k D,\qquad k\ge 1.
\end{equation}
A diffusion-driven (Turing) instability occurs if \eqref{eq:intro_kinetic_stability} holds while at least one nonconstant mode becomes unstable, i.e.
\begin{equation}
\label{eq:intro_turing_condition}
\exists\,k\ge 2:\ \operatorname{Spec}(\mathcal{A}_k)\cap\{\Re z>0\}\neq\emptyset.
\end{equation}
This mode-wise mechanism underlies self-organized pattern formation in a wide range of biological and physical settings, where the interplay between local reactions and environmental signaling plays key roles \citep{cross_pattern_1993,dolnik_resonant_2001,colizza_reactiondiffusion_2007,belik_natural_2011}.
In the classical two-species linearization formulated in Turing's prototype \citep{turing_chemical_1952,cherniha_nonlinear_2017,kondo_reaction-diffusion_2010},
\begin{align}
\label{eq:intro_turing}
\begin{cases}
\partial_{t}X = \dfrac{\mu'}{\rho^2}\partial_{\theta\theta}X + a(X-h)+b(Y-k), \\[4pt]
\partial_{t}Y = \dfrac{\nu'}{\rho^2}\partial_{\theta\theta}Y + c(X-h)+d(Y-k),
\end{cases}
\end{align}
the instability reduces to algebraic sign conditions on the Jacobian and diffusion coefficients; 
extensive refinements address
geometry, boundary spectra, and robustness \citep{marcon_turing_2012,chen_non-linear_2019,maini_turings_2012,ninomiya_example_2024,gierer_theory_1972,kondo_reaction-diffusion_2010,mimura_diffusive_1978}.
Stochastic perturbations enlarge the instability region \citep{cao_stochastic_2014}, and network/generalized couplings embed multiscale feedback
\citep{regueira_lopez_de_garayo_network_2025}. Yet, many realistic biological patterns require additional structure (non-motile components, switching, lineage effects),
e.g.\ zebrafish skin patterning \citep{kondo_studies_2021,frohnhofer_iridophores_2013,mahalwar_local_2014}.

Many biological patterns, however, are not initiated by purely diffusive instabilities: motile populations can bias their movement along weak spatial cues, thereby amplifying heterogeneity through directed transport.
The directed movement substantially shapes subsequent morphogenesis.
Cell motility therefore plays a central role in phenomena such as aggregation of the social amoeba \textit{Dictyostelium discoideum} \citep{hofer_dictyostelium_1995}, stripe formation in fish pigmentation \citep{painter_stripe_1999}, chick gastrulation and limb morphogenesis \citep{li_cell_1999,yang_cell_2002}, various bacterial aggregation behaviors \citep{budrene_complex_1991,budrene_dynamics_1995}, and primitive streak formation \citep{painter_chemotactic_2000}. 
A canonical continuum description is provided by Keller--Segel-type chemotaxis systems \citep{keller_initiation_1970,keller_model_1971}. For a population density $n(x,t)$ and a chemoattractant concentration $c(x,t)$ one considers
\begin{equation}
\label{eq:intro_KS}
\begin{cases}
\partial_{t} n = D_n\Delta n - \nabla \cdot \bigl( n \chi(c)\nabla c \bigr) + f(n,c),\\[2pt]
\partial_{t} c = D_c\Delta c + g(n,c),
\end{cases}
\end{equation}
where the flux $-\nabla\cdot(n\chi(c)\nabla c)$ models biased migration up gradients of $c$.
Depending on dimension, sensitivity, and kinetics, such transport can yield aggregation, patterning, or even finite-time blow-up \citep{murray_mathematical_2002,jager_explosions_1992,horstmann_boundedness_2005,jin_critical_2026,horstmann_1970_2003,arumugam_keller-segel_2021}.
Further refinements enhance mechanism realism \citep{martiel_model_1987,goldbeter_biochemical_1996,monk_cyclic_1989,spiro_model_1997}.
We emphasize that the present work focuses on linear, mode-wise instability of homogeneous equilibria under Neumann boundary conditions (Turing-type mechanisms) rather than blow-up phenomena.

In our hybrid tumor-microenvironment setting the activated microenvironmental state $A(x,t)$ is non-diffusive (a pointwise ODE variable), so its spatial gradient is not, in general, well-defined.
To formulate chemotactic drift rigorously we therefore introduce a diffusive chemoattractant $c(x,t)$ generated by $A$, for instance
\begin{equation}
\label{eq:intro_c_from_A}
\partial_t c = d_c\Delta c + \kappa A - \rho c,
\qquad \partial_{\mathrm n}c=0 \ \text{on }\partial U,
\end{equation}
with $d_c>0$, $\kappa>0$, and $\rho>0$.
In fast-relaxation regimes one may interpret $c\approx (\kappa/\rho)A$ as a quasi-steady approximation, but \eqref{eq:intro_c_from_A} is the standing formulation that guarantees $\nabla c$ is meaningful and keeps the chemotaxis terms well-posed.

Now we consider tumor microenvironment motivation. Spatial structure in tumor microenvironments (TMEs) is shaped by more than diffusion alone \citep{bubba_chemotaxis-based_2019,watts_pdgf-aa_2016,puliafito_three-dimensional_2015,oudin_physical_2016,esfahani_three-dimensional_2022}.
In addition to random dispersal and directed migration along biochemical cues (chemotaxis/haptotaxis) \citep{aznavoorian_signal_1990,taraboletti_thrombospondin-induced_1987}, TMEs exhibit
non-motile state switching of microenvironmental constituents (e.g.\ quiescent/activated stromal states) \citep{arina_tumor-associated_2016,kojima_autocrine_2010,anderberg_paracrine_2009,turlej_cross-talk_2025},
signal/drug modulation on intermediate time scales \citep{issa_dynamics_2025,yu_tumor_2025}, and remodeling of the extracellular matrix (ECM) \citep{henke_extracellular_2020,huang_extracellular_2021,hu_extracellular_2025}
together with mechanical constraints \citep{luu_mechanical_2024} that modulate transport and effective growth.
These features generate spatial heterogeneity through coupled feedback loops: signals and stromal activation reshape the local habitat,
while tumor populations respond by proliferation, phenotypic conversion, and directed movement.
Recent hybrid multiscale studies further highlight that spatial vascular/perfusion heterogeneity can emerge from bidirectional chemotactic feedback and, in turn, create hypoxic regions and drug-resistance niches; see, e.g., \citep{liu_bidirectional_2025} for a PDE--agent-based framework where endothelial--tumor angiogenic factor two-way coupling yields an explicit finite-domain Turing threshold and heterogeneous drug penetration.
From a modeling perspective, this naturally leads to systems in which (i) motile populations are governed by diffusion and transport,
(ii) microenvironmental variables include non-diffusive internal states with local transitions, and (iii) inhibitory or chemoattractive
signals relax rapidly via diffusion and decay.

Motivated by these considerations, we adopt a hybrid PDE--ODE framework: motile populations evolve by parabolic PDEs,
whereas microenvironmental state variables evolve by pointwise (in space) ODEs that encode phenotypic switching.
To formulate chemotactic drift rigorously while keeping the microenvironmental state $A$ non-diffusive, we introduce a diffusive chemoattractant $c(x,t)$ generated by $A$ (and, in feedback settings, possibly by the populations), so that $\nabla c$ is well-defined.
This setup allows us to isolate a directionality principle: when signaling is one-way and damped, transport does not alter the population spectrum,
whereas bidirectional feedback can induce effective cross-diffusion and recover Turing-type instabilities.

\begin{figure}[htbp]
    \centering
    \includegraphics[width=\textwidth]{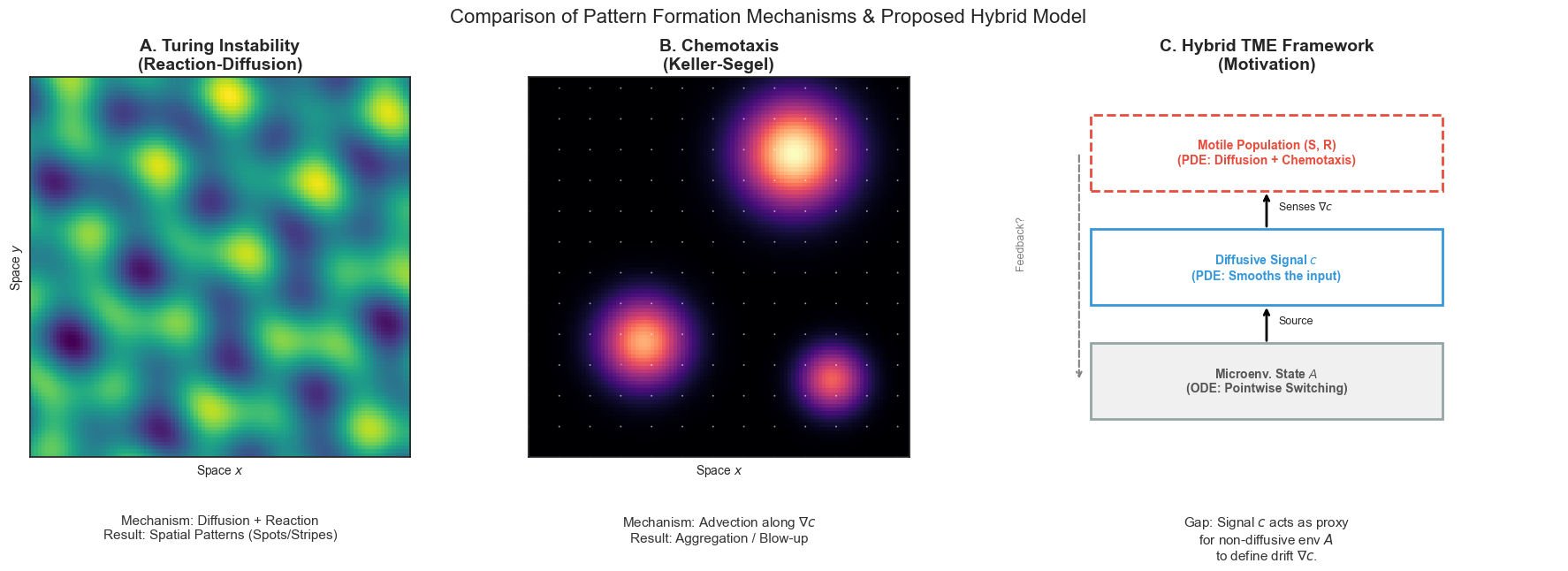} 
    \caption{\textbf{Comparison of pattern formation mechanisms and the proposed hybrid framework.} 
    \textbf{(A)} Classical diffusion-driven (Turing) instability generates spatial patterns from homogeneity (Eq.~\ref{eq:intro_RD_general}). 
    \textbf{(B)} Chemotaxis (Keller--Segel) drives directed transport leading to aggregation or blow-up (Eq.~\ref{eq:intro_KS}). 
    \textbf{(C)} The proposed hybrid TME framework couples motile populations (PDE) to non-motile microenvironmental states ($A$, ODE) via a diffusive signal ($c$) to rigorously define chemotactic drift.}
    \label{fig:mechanisms}
\end{figure}

Despite extensive progress on diffusion-driven (Turing) mechanisms and on Keller--Segel-type chemotactic drift \citep{landge_pattern_2020,arumugam_keller-segel_2021}, a unified mode-wise stability theory that combines reaction--diffusion dispersion relations with chemotaxis-induced transport within hybrid PDE--ODE frameworks featuring non-diffusive microenvironmental state transitions remains limited.
In particular, when motile populations are coupled to non-motile microenvironmental states through fast diffusive signals,
it is not a priori clear how classical Turing criteria interact with directional transport and with the presence (or absence) of feedback. In addition, recent hybrid PDE--ABM studies establish well-posedness and derive continuum limits for multiscale feedback models \citep{wang_analysis_2025}, a unified mode-wise stability theory that transfers reaction--diffusion (Turing) criteria to chemotaxis-driven drift in hybrid PDE--ODE settings remains less developed.

We address this gap by introducing a diffusive chemoattractant $c(x,t)$ (mathematically necessary to ensure $\nabla c$ is well-defined),
generated by microenvironmental activation and, in feedback settings, by the populations, and by organizing the analysis according to a directionality dichotomy: one-way damped coupling, in which $c$ evolves independently of the populations, versus two-way coupling, in which feedback from $(S,R)$ to $c$ induces effective cross-diffusion (in an appropriate fast-relaxation closure)
and can restore Turing-type instabilities.

Throughout, pattern formation refers to linear, mode-wise instability of a spatially homogeneous equilibrium under homogeneous Neumann boundary conditions (a Turing-type mechanism).
We explicitly distinguish this notion from chemotaxis-driven aggregation and blow-up phenomena in Keller--Segel systems, which are not the focus of the present work. Figure~\ref{fig:mechanisms} compares diffusion-driven instability and chemotaxis-driven aggregation.

Our analysis connects a hybrid PDE--ODE reaction--diffusion core model with chemotaxis-driven transport through a directionality-based stability framework (Figure~\ref{fig:stability_dichotomy}). The main results are:
\begin{enumerate}[label=(C\arabic*), leftmargin=*, itemsep=2pt]
\item \textbf{Well-posedness and positivity.}
We prove invariance of the positive cone and global well-posedness of mild solutions for the hybrid PDE--ODE system under standard Neumann compatibility assumptions.

\item \textbf{Long-time reduction and global convergence of kinetics.}
Using exponential decay of the inhibitory signal and pointwise conservation in the non-motile subsystem, we derive a long-time reduction to the limiting $(S,R)$ kinetics and establish global attractivity of the unique coexistence equilibrium.

\item \textbf{No diffusion-driven instability in the base reaction--diffusion subsystem.}
We show that the coexistence equilibrium remains linearly asymptotically stable with respect to all nonconstant Laplacian modes for arbitrary positive diffusion coefficients; hence no classical Turing instability occurs in the base reaction--diffusion dynamics.

\item \textbf{One-way damped chemotaxis suppresses patterns.}
Introducing a diffusive chemoattractant $c(x,t)$ to formulate drift rigorously, we prove that unidirectional, damped coupling (no feedback from the populations to $c$) yields a block-triangular linearization and therefore cannot destabilize the population spectrum; diffusion/chemotaxis-driven patterns are suppressed.

\item \textbf{Two-way feedback recovers instability via explicit mode-wise criteria.}
For bidirectional coupling, we derive explicit trace/determinant conditions characterizing when feedback generates effective cross-diffusion (in a fast-relaxation closure) and restores Turing-type instabilities for some nonconstant Laplacian mode.
\end{enumerate}

\begin{figure}[htbp]
    \centering
    \includegraphics[width=\textwidth]{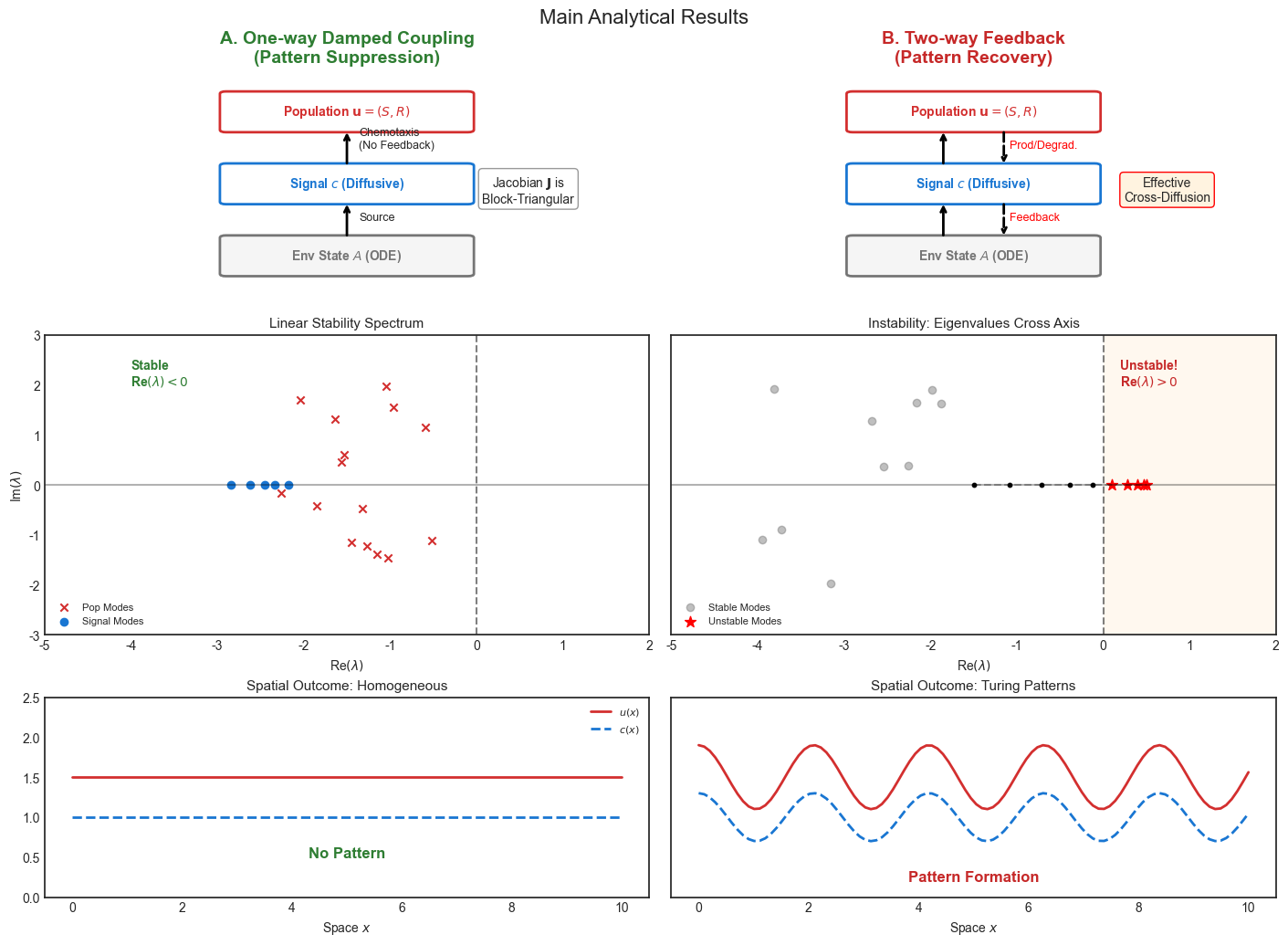} 
    \caption{\textbf{Directionality dichotomy in stability results (Main Analytical Contributions).} 
    \textbf{(Left)} Under one-way damped coupling, the Jacobian is block-triangular, maintaining the stability of the population spectrum (all eigenvalues $\Re \lambda < 0$). Spatially, the system remains homogeneous.
    \textbf{(Right)} Under two-way feedback, the interaction generates effective cross-diffusion, causing eigenvalues to cross the imaginary axis ($\Re \lambda > 0$). This leads to the recovery of spatial Turing patterns.}
    \label{fig:stability_dichotomy}
\end{figure}

The paper is organized as follows. Section~\ref{sec:model_prelim} introduces the hybrid PDE--ODE model, fixes standing assumptions and notation, and records key structural identities (positivity, pointwise conservation, and decay of the inhibitory signal). Section~\ref{sec:main_results} states the main analytical results on well-posedness, long-time reduction, global convergence of the coexistence equilibrium, and the absence of diffusion-driven instability in the base reaction--diffusion subsystem.
In this section, we also introduce chemotaxis extensions formulated via a diffusive chemoattractant $c$ and establish the directionality dichotomy between one-way damped coupling (suppression) and two-way feedback (recovery), culminating in explicit mode-wise instability criteria. 
Section~\ref{sec:analysis} contains the corresponding proofs. Section~\ref{sec:discussion} concludes with biological interpretation, limitations, and directions for future work.

\section{Model formulation and standing assumptions} \label{sec:model_prelim}

\subsection{Hybrid PDE--ODE core model}
\label{subsec:hybrid_core_model}

Let $U\subseteq\mathbb{R}^N$ ($N\in\{1,2,3\}$) be a bounded domain with smooth boundary $\partial U$.
We consider the state vector
\[
\mathbf{u}(x,t) \coloneqq (S(x,t),R(x,t),D(x,t),P(x,t),A(x,t))^{\mathsf T},
\qquad (x,t)\in U\times(0,\infty),
\]
where $S$ and $R$ denote two competing motile populations, $D$ is a fast diffusing inhibitory (or drug-like) signal,
and $(P,A)$ represent non-motile microenvironmental states (passive/active) undergoing local phenotypic switching.
A defining feature of the model is its hybrid structure: $S,R,D$ are governed by parabolic PDEs, whereas $P$ and $A$
evolve by pointwise (in space) ODEs and therefore carry no boundary conditions. Figure~\ref{fig:diffusion} illustrates representative diffusing and non-diffusing components in the spatial domain $\Omega$.

\begin{figure}[hbtp]
    \centering
    \includegraphics[width=0.75\linewidth]{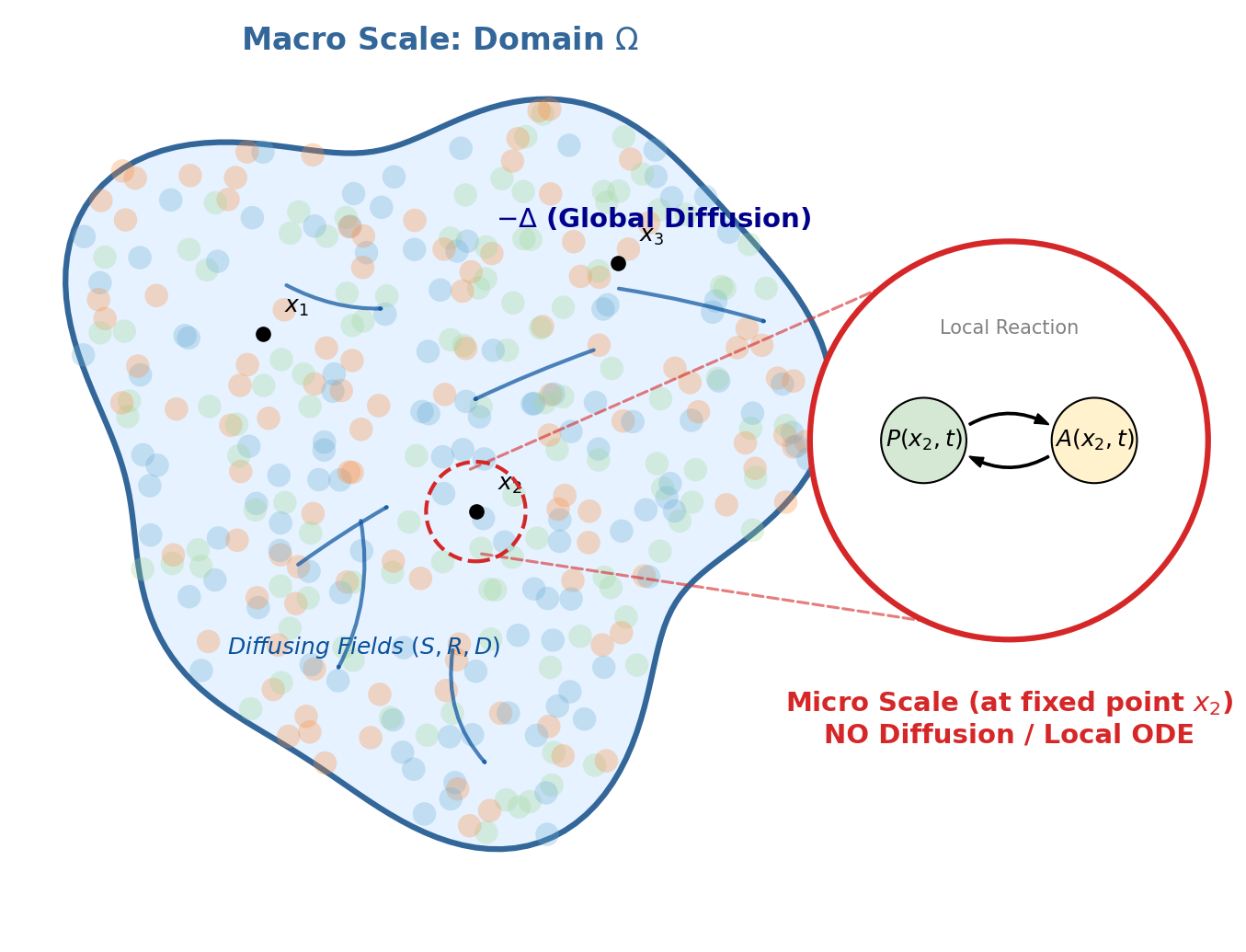}
    \caption{\textbf{The diffusive and non-diffusive components.}}
    \label{fig:diffusion}
\end{figure}

The core system reads
\begin{align}
\label{eq:full_system}
\begin{cases}
\partial_{t} S
= d_{S}\Delta S
+ \lambda_{S} S \left(1 - \tfrac{S+R}{K}\right)
- \alpha S
- \delta(D)S
+ \xi \bigl[1 - \phi(D)\bigr]R,
& x\in U,\ t>0, \\[4pt]
\partial_{t} R
= d_{R}\Delta R
+ \lambda_{R} R \left(1 - \tfrac{S+R}{K}\right)
+ \alpha S
+ \eta \phi(D)\,A R
- \xi \bigl[1 - \phi(D)\bigr]R,
& x\in U,\ t>0, \\[4pt]
\partial_{t} D
= d_{D}\Delta D
- \gamma_{d}D,
& x\in U,\ t>0, \\[4pt]
\partial_{t} P
= -\theta \phi(D)P
+ \beta \bigl[1 - \phi(D)\bigr]A,
& x\in U,\ t>0, \\[4pt]
\partial_{t} A
= \theta \phi(D)P
- \beta \bigl[1 - \phi(D)\bigr]A,
& x\in U,\ t>0.
\end{cases}
\end{align}
Here $(S,R)$ undergo logistic growth with shared carrying capacity $K$ and interconversion at rates $\alpha$ and $\xi$.
The signal $D$ suppresses $S$ via a saturating inhibition term and simultaneously modulates microenvironmental activation through $\phi(D)$.
Importantly, the $D$-equation is a decoupled damped diffusion equation, so that $D$ admits maximum-principle bounds and decays exponentially in time.
The non-motile subsystem $(P,A)$ is a two-state switching model driven by the local value $D(x,t)$; since it has no spatial derivatives, it is an ODE at each spatial point. Figure~\ref{fig:TME_interaction} shows the interaction structure of each component in \eqref{eq:full_system}.

\begin{figure}[hbtp]
    \centering
    \includegraphics[width=0.75\linewidth]{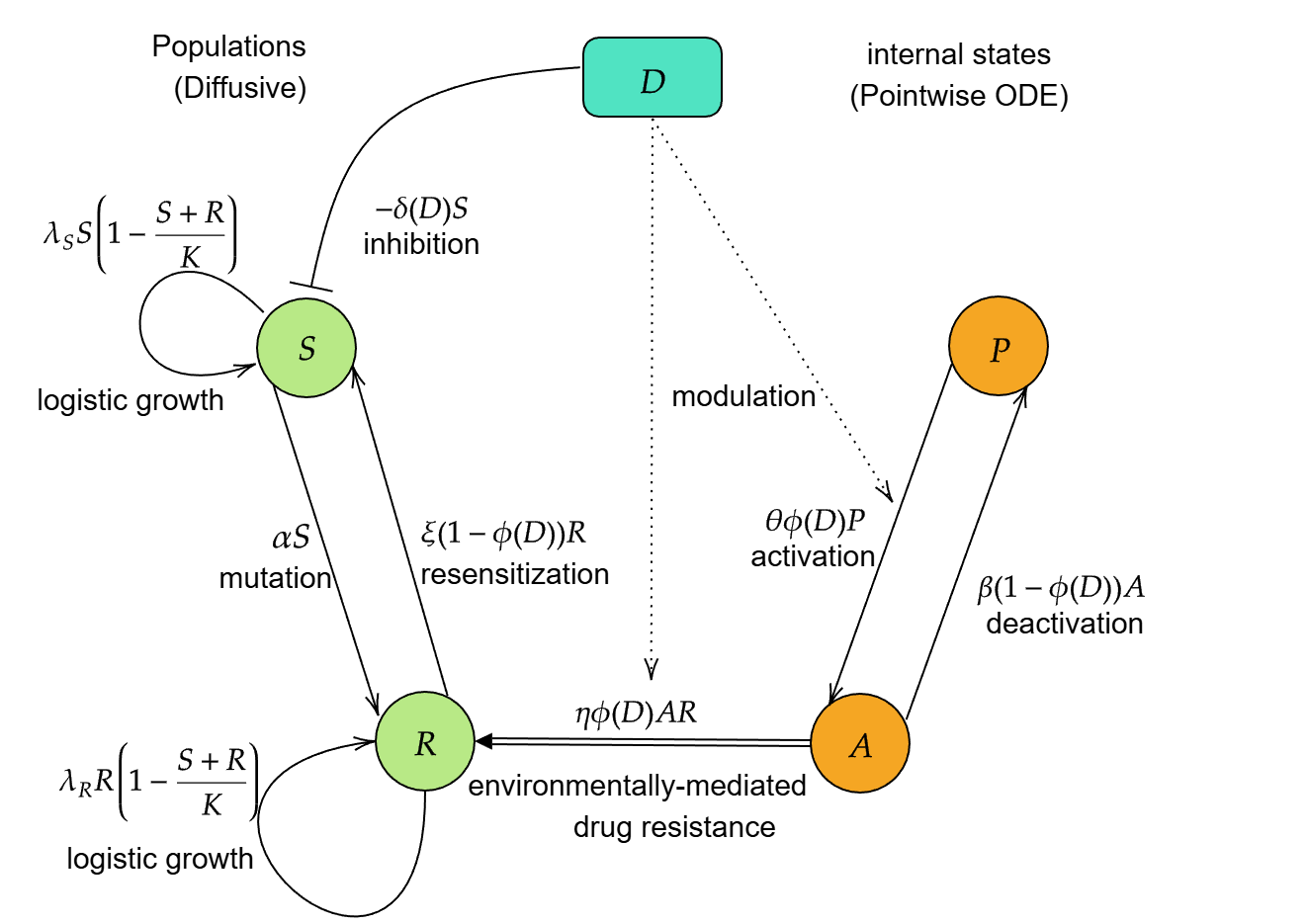}
    \caption{\textbf{The model structure and their interactions.}}
    \label{fig:TME_interaction}
\end{figure}

We fix the inhibition and activation functions as
\begin{equation}
\label{eq:delta_def}
\delta(D) \coloneqq \delta_{0}\frac{D}{D+K_{D}},
\qquad D\ge 0,
\end{equation}
and
\begin{equation}
\label{eq:phi_def}
\phi:[0,\infty)\to [0,1)\quad \text{is $C^{1}$, nondecreasing, globally Lipschitz, and satisfies }\ \phi(0)=0.
\end{equation} 
Figure~\ref{fig:phi_delta} shows the graphs of $\phi$ and $\delta$.
All parameters are assumed strictly positive unless stated otherwise (see Table~\ref{tab:parameters}).

\begin{figure}[htbp]
    \centering
    \includegraphics[width=0.75\linewidth]{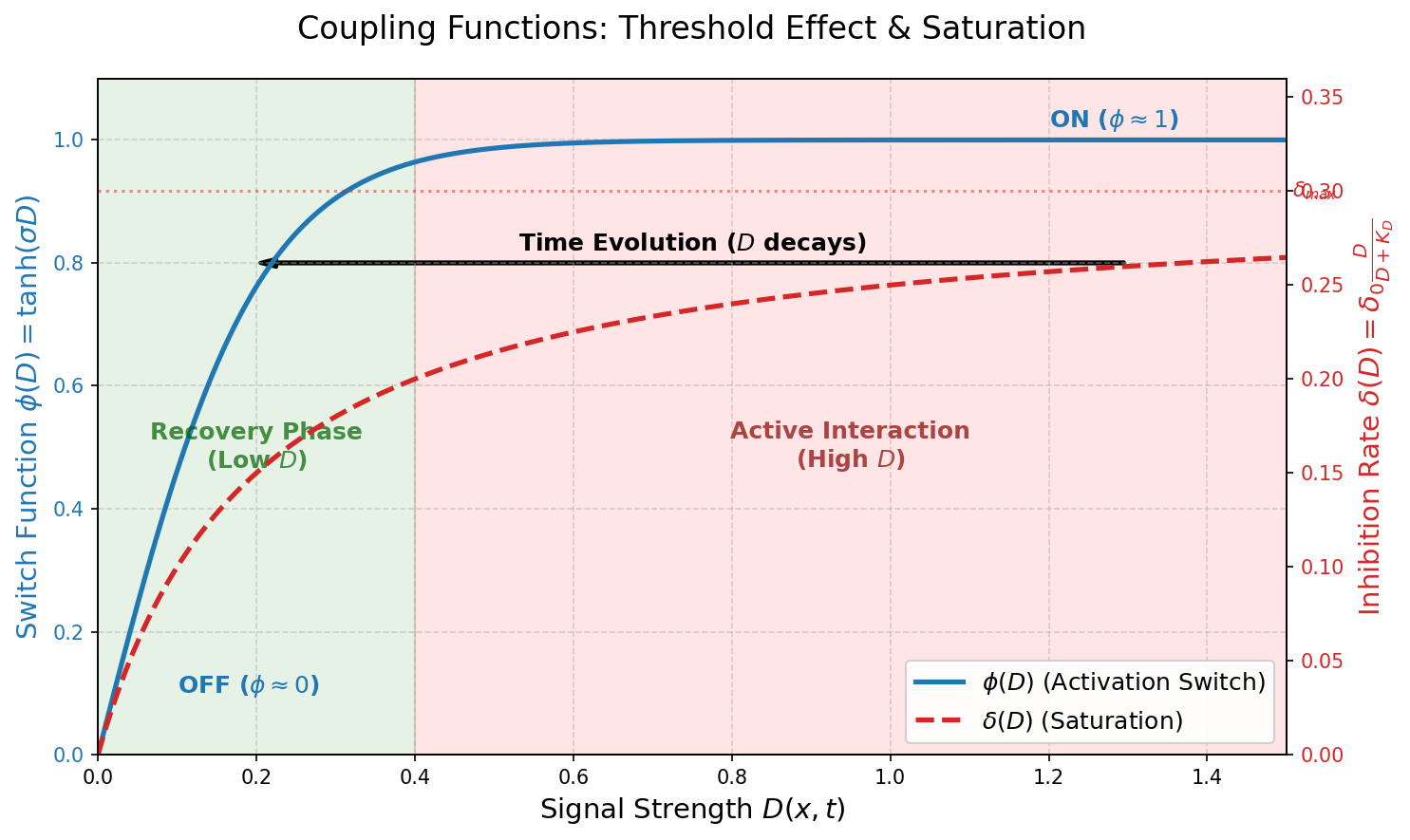}
    \caption{\textbf{Graphs of $\phi$ and $\delta$.}}
    \label{fig:phi_delta}
\end{figure}

\begin{table}[htbp]
\centering
\caption{Variables and parameters in the hybrid PDE--ODE model \eqref{eq:full_system}. Typical units are indicative.}
\label{tab:parameters}
\renewcommand{\arraystretch}{1.15}
\begin{tabular}{lll}
\hline
Symbol & Interpretation & Typical unit\\
\hline
$S(x,t)$ & population 1 (e.g.\ sensitive phenotype) & density\\
$R(x,t)$ & population 2 (e.g.\ resistant phenotype) & density\\
$D(x,t)$ & diffusive inhibitory signal & concentration\\
$P(x,t)$ & passive (quiescent) non-motile component & density \\
$A(x,t)$ & active non-motile component & density \\
\hline
$d_S,d_R$ & diffusion coefficients of $S,R$ & length$^2$/time\\
$d_D$ & diffusion coefficient of $D$ & length$^2$/time (fast)\\
$\lambda_S,\lambda_R$ & intrinsic growth rates & 1/time\\
$K$ & carrying capacity shared by $S,R$ & density\\
$\alpha$ & conversion rate $S\to R$ & 1/time\\
$\xi$ & conversion rate $R\to S$ (modulated by $D$) & 1/time\\
$\eta$ & promotion of $R$ by active component $A$ & 1/(density$\cdot$time)\\
$\delta_0$ & maximal inhibition rate of $S$ by $D$ & 1/time\\
$K_D$ & half-saturation constant in $\delta(D)$ & concentration\\
$\gamma_d$ & decay/clearance rate of $D$ & 1/time\\
$\theta$ & activation rate $P\to A$ (via $\phi(D)$) & 1/time\\
$\beta$ & deactivation rate $A\to P$ (via $1-\phi(D)$) & 1/time\\
$\phi(D)$ & activation function & $\in[0,1)$, smooth saturating\\
\hline
\end{tabular}
\end{table}

Initial and boundary conditions. We prescribe nonnegative initial data
\begin{equation}
\begin{aligned}
\label{eq:IC_core}
&S(\cdot,0)=S_0,\quad R(\cdot,0)=R_0,\quad D(\cdot,0)=D_0,\\[6pt] &P(\cdot,0)=P_0,\quad A(\cdot,0)=A_0, \quad \mathbf{u}_0\ge 0\ \text{a.e.\ in }U,
\end{aligned}
\end{equation}
with $\mathbf{u}_0\in (L^\infty(U))^5$ (and higher regularity when classical solutions are invoked).
Homogeneous Neumann boundary conditions are imposed only on the diffusive variables $(S,R,D)$:
\begin{equation}
\label{eq:BC_core}
\partial_{\mathrm n}S=\partial_{\mathrm n}R=\partial_{\mathrm n}D=0
\qquad \text{on }\partial U\times(0,\infty),
\end{equation}
where $\partial_{\mathrm n}$ denotes the outward normal derivative.
No boundary conditions are imposed on $P$ and $A$ since they satisfy pointwise ODEs.

\begin{remark}[Chemotaxis extensions]
In Section~\ref{subsec:main_directionality} we extend \eqref{eq:full_system} by introducing chemotactic drift for $S$ and $R$.
Because $A$ is non-diffusive, its gradient need not exist; we therefore formulate chemotaxis via a diffusive chemoattractant $c(x,t)$
driven by $A$ (and, in feedback settings, by the populations), so that $\nabla c$ is well-defined under \eqref{eq:BC_core}.    
\end{remark}

\subsection{Structural identities}
\label{subsec:structural_identities}

We record two elementary structural properties of \eqref{eq:full_system} that will be used repeatedly in the well-posedness
theory and in the long-time reduction argument. The first is a pointwise conservation law for the non-diffusive subsystem
$(P,A)$. The second is maximum-principle control and exponential energy decay for the decoupled damped diffusion equation
satisfied by the signaling field $D$.

\begin{lemma}[Pointwise conservation in the non-diffusive $(P,A)$ subsystem]
\label{lem:PA_conservation}
Let $(P,A)$ solve the last two equations of \eqref{eq:full_system} for a given (measurable) signal $D(x,t)$.
Then for a.e.\ $x\in U$ and all $t\ge 0$,
\begin{equation}
\label{eq:PA_conservation}
P(x,t)+A(x,t)=P_0(x)+A_0(x).
\end{equation}
In particular, if $P_0,A_0\ge 0$ a.e.\ in $U$, then $P(x,t),A(x,t)\ge 0$ for all $t\ge 0$, and
\begin{equation}
\label{eq:PA_Linfty_bound}
\sup_{t\ge0}\bigl(\|P(t)\|_{L^\infty(U)}+\|A(t)\|_{L^\infty(U)}\bigr)
\le \|P_0\|_{L^\infty(U)}+\|A_0\|_{L^\infty(U)}.
\end{equation}
\end{lemma}

\begin{proof}
Fix $x\in U$. Adding the $P$- and $A$-equations in \eqref{eq:full_system} yields
$\partial_t(P(x,t)+A(x,t))=0$, hence \eqref{eq:PA_conservation}.
Nonnegativity follows from quasi-positivity of the linear ODE system for $(P(x,\cdot),A(x,\cdot))$, and
\eqref{eq:PA_Linfty_bound} follows directly from \eqref{eq:PA_conservation}.
\end{proof}

\begin{proposition}[Maximum principle and exponential energy decay for the signal $D$]
\label{prop:D_decay}
Let $D$ solve the decoupled damped diffusion equation
\[
\partial_t D = d_D\Delta D - \gamma_d D \quad \text{in }U\times(0,\infty),\qquad
\partial_{\mathrm n}D=0 \ \text{on }\partial U\times(0,\infty),\qquad
D(\cdot,0)=D_0,
\]
with $d_D>0$, $\gamma_d>0$, and $D_0\in L^\infty(U)$, $D_0\ge0$ a.e.\ in $U$.
Then:
\begin{enumerate}[label=(\roman*), leftmargin=*]
\item (\textbf{$L^\infty$ bound}) For all $t\ge 0$,
\begin{equation}
\label{eq:D_Linfty_bound}
0\le D(x,t)\le \|D_0\|_{L^\infty(U)} \qquad \forall x\in U.
\end{equation}
\item (\textbf{$L^2$ energy identity}) For all $t>0$,
\begin{equation}
\label{eq:D_energy_identity}
\frac{1}{2}\frac{d}{dt}\|D(t)\|_{L^2(U)}^2
+d_D\|\nabla D(t)\|_{L^2(U)}^2+\gamma_d\|D(t)\|_{L^2(U)}^2=0.
\end{equation}
\item (\textbf{Exponential decay}) Consequently, for all $t\ge0$,
\begin{equation}
\label{eq:D_L2_decay}
\|D(t)\|_{L^2(U)}\le e^{-\gamma_d t}\|D_0\|_{L^2(U)}.
\end{equation}
\end{enumerate}
In particular, by the structural assumptions on $\phi$ and $\delta$,
\begin{equation}
\label{eq:phi_delta_vanish}
\phi(D(\cdot,t))\to0,\qquad \delta(D(\cdot,t))\to0
\quad \text{as }t\to\infty,
\end{equation}
in any topology in which $D(\cdot,t)\to0$.
\end{proposition}

\begin{proof}
The bound \eqref{eq:D_Linfty_bound} follows from the parabolic maximum principle for the damped heat equation with Neumann boundary conditions and nonnegative initial data.
Identity \eqref{eq:D_energy_identity} is obtained by multiplying the equation by $D$, integrating over $U$, and integrating by parts using $\partial_{\mathrm n}D=0$.
Estimate \eqref{eq:D_L2_decay} follows immediately from \eqref{eq:D_energy_identity} by Gr\"onwall's inequality.
\end{proof}

\begin{remark}[Spectral representation]
When an explicit solution structure for $D$ is needed (e.g.\ for mode-wise arguments), $D$ admits a Neumann eigenfunction expansion.
We collect the spectral representation and the required eigenfunction estimates in Appendix~\ref{app:spectral}.    
\end{remark}

\subsection{Functional-analytic setting and solution concepts}
\label{subsec:functional_setting}

We briefly specify the function spaces and solution concepts used for the hybrid PDE--ODE system \eqref{eq:full_system}.
The diffusive components $(S,R,D)$ are treated by semilinear parabolic theory under homogeneous Neumann boundary conditions,
while the non-diffusive variables $(P,A)$ solve pointwise ODEs in space driven by $D(x,t)$.

\subsubsection{Spaces and notation.}
Let $T>0$ and write $U_T\coloneqq U\times(0,T]$.
We denote by $(\cdot,\cdot)$ the $L^2(U)$ inner product and by $\langle\cdot,\cdot\rangle$ the duality pairing between
$H^1(U)$ and $H^1(U)'$.
For the parabolic variables under Neumann boundary conditions we use the standard energy space
\begin{equation}
\label{eq:V_space}
V_{\mathfrak{N}}(T)\coloneqq L^2(0,T;H^1(U))\cap H^1(0,T;H^1(U)').
\end{equation}

\subsubsection{Abstract form and analytic semigroup.}
Let $\Delta_{\mathfrak N}$ denote
the Laplacian on $U$ with homogeneous Neumann boundary conditions.
We work on the Banach lattice
\begin{align*}
X &\coloneqq \bigl(C(\overline U)\bigr)^3 \times \bigl(L^\infty(U)\bigr)^2,
\\[6pt]
X_+ &\coloneqq \{\mathbf u\in X:\ u_i\ge 0\ \text{(pointwise for $i=1,2,3$ and a.e.\ for $i=4,5$)}\}.
\end{align*}
Write $\mathbf u=(S,R,D,P,A)^{\mathsf T}$ and set
\[
\mathcal D \coloneqq \mathrm{diag}(d_S,d_R,d_D,0,0).
\]
Define the linear operator $\mathcal A:\mathcal D(\mathcal A)\subset X\to X$ by
\begin{equation*}
\begin{aligned}
\mathcal A\mathbf u &\coloneqq \mathcal D(-\Delta_{\mathfrak N}+I)\mathbf u, \\[4pt]
\mathcal D(\mathcal A) &\coloneqq
\Bigl\{(S,R,D,P,A)\in \bigl(C^2(\overline U)\bigr)^3\times \bigl(L^\infty(U)\bigr)^2:
\ \partial_{\mathrm n}S=\partial_{\mathrm n}R=\partial_{\mathrm n}D=0\Bigr\}.
\end{aligned}
\end{equation*}
Then $-\mathcal A$ generates the positive analytic semigroup $\{\mathbb P(t)\}_{t\ge 0}$ on $X$ given by
\begin{equation}\label{eq:semigroup_def_shifted}
\mathbb P(t)\coloneqq
\operatorname{diag}\!\Bigl(e^{d_S t(\Delta_{\mathfrak N}-I)},\ e^{d_R t(\Delta_{\mathfrak N}-I)},\
e^{d_D t(\Delta_{\mathfrak N}-I)},\ I,\ I\Bigr),\qquad t\ge 0,
\end{equation}
where $I$ denotes the identity on $L^\infty(U)$.

We rewrite \eqref{eq:full_system} in the abstract form
\begin{equation}\label{eq:abstract_shifted}
\partial_t \mathbf u + \mathcal A \mathbf u = \mathcal F(\mathbf u),
\qquad
\mathcal F(\mathbf u)\coloneqq \mathcal D\mathbf u + \mathcal G(\mathbf u),
\end{equation}
where $\mathcal G$ collects the reaction terms in \eqref{eq:full_system}.

\begin{definition}[Mild solution]\label{def:mild_sol}
Let $\mathbf u_0\in X$ and $T\in(0,\infty]$.
A function $\mathbf u:[0,T)\to X$ is called a mild solution of \eqref{eq:full_system} on $[0,T)$ if
$\mathbf u\in C([0,T);X)$, $\mathcal F(\mathbf u(\cdot))\in L^1((0,t);X)$ for all
$t<T$, and
\begin{equation}\label{eq:mild_formula_shifted}
\mathbf u(t)=\mathbb P(t)\mathbf u_0 + \int_0^t \mathbb P(t-s)\,\mathcal F(\mathbf u(s))\,ds,
\qquad t\in[0,T),
\end{equation}
where the integral is a Bochner integral in $X$.
\end{definition}

\begin{remark}[Unshifted form]\label{rem:unshifted_equiv}
Formulation \eqref{eq:abstract_shifted} is equivalent to
$\partial_t\mathbf u-\mathcal D\Delta_{\mathfrak N}\mathbf u=\mathcal G(\mathbf u)$ on $X$.
In particular, one may alternatively use the semigroup
$\operatorname{diag}(e^{d_S t\Delta_{\mathfrak N}},e^{d_R t\Delta_{\mathfrak N}},e^{d_D t\Delta_{\mathfrak N}},I,I)$;
the two choices differ only by whether the linear term $-\mathcal D\mathbf u$ is kept on the left
or absorbed into $\mathcal F$.
\end{remark}

\subsection{Chemotaxis extensions with a diffusive signal $c$}
\label{subsec:chemotaxis_extensions_overview}

Chemotactic drift terms of Keller--Segel type require the spatial gradient of a signal field.
In the hybrid core model \eqref{eq:full_system}, however, the activated microenvironmental state $A(x,t)$ is non-diffusive
and evolves by a pointwise ODE; in general one only has $A(\cdot,t)\in L^\infty(U)$, so $\nabla A$ is not well-defined.
To formulate directed migration rigorously while preserving the interpretation of $A$ as a non-motile state variable, we
introduce a diffusive chemoattractant $c(x,t)$ and write chemotaxis in the form
\[
-\chi_S\nabla\cdot(S\nabla c),\qquad -\chi_R\nabla\cdot(R\nabla c),
\]
with homogeneous Neumann boundary conditions $\partial_{\mathrm n}c=0$ on $\partial U$.
The signal $c$ may be generated by the microenvironmental activation $A$ (one-way coupling) and, in feedback settings, may also
depend on the population densities (two-way coupling). The resulting formulations allow a clean separation between
directional coupling regimes and their mode-wise stability consequences.

\subsubsection{Template I: one-way damped coupling (no feedback from $(S,R)$ to $c$).}
A prototypical unidirectional chemotaxis extension of \eqref{eq:full_system} takes the form
\begin{equation}
\label{eq:chemo_oneway_template}
\begin{cases}
\partial_t S = d_S\Delta S - \chi_S\nabla\cdot(S\nabla c) + f_S(S,R),\\[4pt]
\partial_t R = d_R\Delta R - \chi_R\nabla\cdot(R\nabla c) + f_R(S,R),\\[4pt]
\partial_t c = d_c\Delta c + \kappa A - \rho c,
\end{cases}
\qquad \partial_{\mathrm n}S=\partial_{\mathrm n}R=\partial_{\mathrm n}D=\partial_{\mathrm n}c=0,
\end{equation}
where $(D,P,A)$ evolve as in \eqref{eq:full_system} and $(f_S,f_R)$ denote the corresponding reaction terms in the $S$- and
$R$-equations (e.g.\ the right-hand sides in \eqref{eq:full_system}).
Here $c$ is produced by $A$ and decays at rate $\rho>0$, so perturbations in $c$ are damped and do not receive feedback from
$(S,R)$. For robustness, we also consider the abstract damped signal model
\begin{equation*}
\partial_t c = d_c\Delta c + \mathcal{Q}(c), \qquad \mathcal{Q}'(c^*)<0,
\end{equation*}
which isolates the role of one-way damping in the linearization.

\subsubsection{Template II: two-way coupling (feedback from $(S,R)$ to $c$).}
To capture positive feedback loops characteristic of chemotaxis-driven pattern formation, we also study
bidirectional couplings in which the signal depends on the populations:
\begin{equation}
\label{eq:chemo_twoway_template}
\begin{cases}
\partial_t S = d_S\Delta S - \chi_S\nabla\cdot(S\nabla c) + f_S(S,R),\\[4pt]
\partial_t R = d_R\Delta R - \chi_R\nabla\cdot(R\nabla c) + f_R(S,R),\\[4pt]
\partial_t c = d_c\Delta c + q\,c + h(S,R), \qquad q<0,
\end{cases}
\qquad \partial_{\mathrm n}S=\partial_{\mathrm n}R=\partial_{\mathrm n}c=0,
\end{equation}
where $h(S,R)$ encodes signal production (or activation) by the populations and $q<0$ enforces linear damping of $c$.

\subsubsection{Fast-relaxation closure used in the mode-wise criteria.}
For the purpose of explicit mode-wise instability criteria, it is convenient to consider a fast-relaxation regime for
the signal, e.g.\ $\varepsilon\,\partial_t c=d_c\Delta c+q\,c+h(S,R)$ with $\varepsilon\ll 1$.
Formally, as $\varepsilon\to0$ one obtains a quasi-steady closure
\[
c \approx \mathcal{L}^{-1}\!\bigl(-h(S,R)\bigr), \qquad \mathcal{L}\coloneqq d_c\Delta + q,
\]
which, under local approximations (or parameter limits), yields an effective reduced dependence $c\approx g(S,R)$.
This fast-relaxation closure is justified in Appendix~\ref{app:closure}.
In Section~\ref{subsec:main_directionality} we use this reduced closure to derive explicit trace/determinant conditions for the instability
of individual Neumann Laplacian modes; the one-way case is treated separately and leads to a block-triangular linearization
that suppresses pattern formation.

\section{Main results} \label{sec:main_results}

\subsection{Well-posedness and positivity}
\label{subsec:main_pos_wellposed}

We first establish well-posedness and preservation of nonnegativity for the hybrid PDE--ODE system \eqref{eq:full_system}. We postpone all proofs in Section~\ref{sec:main_results} to Section~\ref{sec:analysis}.

\begin{lemma}[Local Lipschitz continuity of the reaction map]
\label{lem:local_lipschitz}
Assume \eqref{eq:delta_def} and \eqref{eq:phi_def}. Then the map
$\mathcal{G}:\mathbb{R}^5\to\mathbb{R}^5$ defined by the right-hand side of \eqref{eq:full_system}
is locally Lipschitz continuous. Equivalently, for each $M>0$ there exists $L_M>0$ such that for all
$\mathbf{z},\mathbf{y}\in[-M,M]^5$,
\[
|\mathcal{G}(\mathbf{z})-\mathcal{G}(\mathbf{y})|\le L_M\,|\mathbf{z}-\mathbf{y}|.
\]
\end{lemma}

\begin{theorem}[Nonnegativity]\label{thm:positivity}
Assume \eqref{eq:delta_def}--\eqref{eq:phi_def} and let $\mathbf u=(S,R,D,P,A)$ be a classical solution of
\eqref{eq:full_system}--\eqref{eq:BC_core} on its maximal interval of existence $[0,T_{\max})$.
If the initial data satisfy \eqref{eq:IC_core}, then
\[
S(x,t),\ R(x,t),\ D(x,t),\ P(x,t),\ A(x,t)\ge 0
\qquad\text{for all }(x,t)\in \overline U\times[0,T_{\max}).
\]
\end{theorem}

\begin{theorem}[Global existence and \texorpdfstring{$L^\infty$}{L-infinity} bounds]
\label{thm:global_wellposed}
Assume \eqref{eq:delta_def}--\eqref{eq:phi_def}, homogeneous Neumann boundary conditions \eqref{eq:BC_core},
and $\mathbf{u}_0\in\mathbf{X}_+$.
Then the mild solution of \eqref{eq:full_system} exists globally in time, i.e.\ $T_{\max}=\infty$, and is unique in
$C([0,\infty);\mathbf{X})$.
Moreover, for all $t\ge0$,
\begin{align}
\label{eq:global_bounds_D_PA}
\begin{aligned}
&0\le D(x,t)\le \|D_0\|_{L^\infty(U)}\quad \forall x\in U, \\[4pt]
&\sup_{t\ge0}\bigl(\|P(t)\|_{L^\infty(U)}+\|A(t)\|_{L^\infty(U)}\bigr)
\le \|P_0\|_{L^\infty(U)}+\|A_0\|_{L^\infty(U)}.
\end{aligned}
\end{align}
In addition, for each $T>0$ there exists $C_T<\infty$ such that
\begin{equation}
\label{eq:global_bounds_SR}
\sup_{0\le t\le T}\bigl(\|S(t)\|_{L^\infty(U)}+\|R(t)\|_{L^\infty(U)}\bigr)\le C_T.
\end{equation}
\end{theorem}

\begin{remark}[Classical solutions]\citep{ladyzenskaja_linear_1968,rothe_global_1984,gilbarg_elliptic_2001}
\label{rem:classical_solutions}
Under standard H\"older regularity and Neumann compatibility conditions on $(S_0,R_0,D_0)$ (and H\"older regularity of $(P_0,A_0)$),
the global mild solution in Theorem~\ref{thm:global_wellposed} upgrades to a global classical solution by parabolic regularity.
We do not repeat the classical assumptions here; they are stated explicitly whenever classical regularity is invoked.
\end{remark}

\subsection{Long-time reduction and global attractivity of the limiting kinetics}
\label{subsec:main_longtime}

A key structural feature of \eqref{eq:full_system} is that the signaling field $D$ solves a decoupled damped heat equation,
hence decays exponentially, while the non-motile subsystem satisfies a pointwise conservation law.
Consequently, all $D$-mediated terms vanish asymptotically and the dynamics approach a limiting $(S,R)$-system obtained by
setting $D\equiv0$, $\phi(D)\equiv 0$, and $\delta(D)\equiv0$ in the population equations.
This motivates the reduced kinetics
\begin{equation}
\label{eq:SR_reduced}
\begin{cases}
\dfrac{dS}{dt} = f_{S}(S,R)\coloneqq \lambda_{S}S\Bigl(1-\dfrac{S+R}{K}\Bigr)-\alpha S+\xi R, \\[8pt]
\dfrac{dR}{dt} = f_{R}(S,R)\coloneqq \lambda_{R}R\Bigl(1-\dfrac{S+R}{K}\Bigr)+\alpha S-\xi R,
\end{cases}
\end{equation}
which provides the reference homogeneous equilibrium for the mode-wise stability analysis in subsequent sections.

\begin{theorem}[Long-time reduction to the limiting $(S,R)$-kinetics]
\label{thm:longtime_reduction}
Assume \eqref{eq:delta_def}--\eqref{eq:phi_def} and let $\mathbf{u}=(S,R,D,P,A)$ be the global mild solution of
\eqref{eq:full_system} given by Theorem~\ref{thm:global_wellposed}, with $\mathbf{u}_0\in\mathbf{X}_+$.
Then:
\begin{enumerate}[label=(\roman*), leftmargin=*]
\item The signal decays exponentially in $L^2(U)$:
\[
\|D(t)\|_{L^2(U)}\le e^{-\gamma_d t}\|D_0\|_{L^2(U)}\qquad \forall t\ge0.
\]
In particular, $\phi(D(\cdot,t))\to0$ and $\delta(D(\cdot,t))\to0$ as $t\to\infty$ in any topology in which $D(\cdot,t)\to0$.
\item Consequently, the $D$-dependent terms in the $(S,R)$-equations are asymptotically vanishing:
for every $T>0$,
\begin{equation*}
\begin{aligned}
\sup_{t\ge T}\bigg(
&\|\delta(D(t))S(t)\|_{L^\infty(U)}
+\|(1-\phi(D(t)))-1\|_{L^\infty(U)}\,\|R(t)\|_{L^\infty(U)} \\[4pt]
+&\|\phi(D(t))A(t)\|_{L^\infty(U)}\,\|R(t)\|_{L^\infty(U)}
\Big)\longrightarrow 0
\end{aligned}
\end{equation*}
as $T\to\infty$.
\end{enumerate}
In particular, the long-time dynamics of \eqref{eq:full_system} are governed by the limiting kinetics \eqref{eq:SR_reduced}.
\end{theorem}

\begin{theorem}[Global attractivity of the coexistence equilibrium of \eqref{eq:SR_reduced}]
\label{thm:global_attract_reduced}
Assume $K>0$ and $\alpha+\xi>0$. Then the reduced kinetics \eqref{eq:SR_reduced} admit a unique coexistence equilibrium
\begin{equation}
\label{eq:Kstar_reduced}
\mathbb{K}^{*}=(S^{*},R^{*})
=\left(\frac{\xi K}{\alpha+\xi},\ \frac{\alpha K}{\alpha+\xi}\right),
\end{equation}
which lies in the positive quadrant. Moreover, $\mathbb{K}^{*}$ is globally asymptotically stable for all initial data
$(S(0),R(0))$ with $S(0)\ge0$, $R(0)\ge0$, and $S(0)+R(0)>0$.
\end{theorem}

The reduction strategy pursued here---combining dissipative/decoupled components with balance-type identities to obtain an effective low-dimensional kinetics---is in the spirit of recent structured population analyses; cf.\ \citep{liang_global_2025} for a balance-law reduction and subsequent stability classification.
Figure~\ref{fig:phase_portrait} shows the phase portrait of \eqref{eq:SR_reduced}. The vector fields of \eqref{eq:SR_reduced} point toward $\mathbb K^*$, supporting the global stability claim in Theorem~\ref{thm:global_attract_reduced}.

\begin{figure}[htbp]
    \centering
    \includegraphics[width=0.75\linewidth]{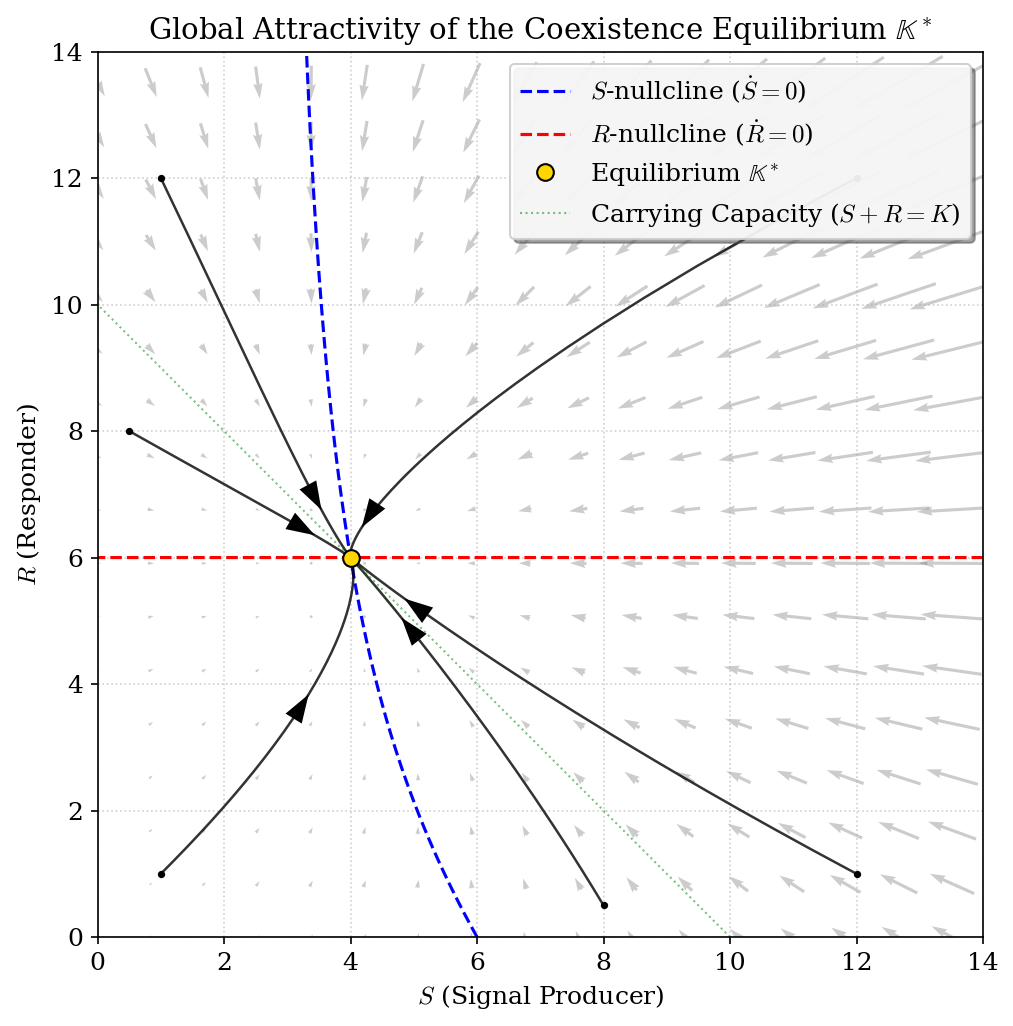}
    \caption{\textbf{Phase portrait of \eqref{eq:SR_reduced}}. The blue and red dashed lines represent the $S$- and $R$-nullclines, respectively. The green dashed line represents the one-dimensional invariant manifold $S+R=K$, where the tumor population reaches the carrying capacity $K$. The black solid curves are representative trajectories starting from different initial positions. The gray arrows indicate the direction and relative magnitude of the vector field of \eqref{eq:SR_reduced}, and they all point toward the unique equilibrium $\mathbb K^*$ (yellow dot), confirming the global stability of $\mathbb{K}^*$. Parameter values include $K=10$, $\alpha=0.6$, $\xi=0.4$, $\lambda_S=1.5$, $\lambda_R=1.0$.}
    \label{fig:phase_portrait}
\end{figure}

\begin{remark}[Equilibrium composition]
At the coexistence equilibrium \eqref{eq:Kstar_reduced},
\[
\frac{S^*}{R^*}=\frac{\xi}{\alpha},
\]
so the long-time composition is determined by the relative conversion rates.
\end{remark}

\subsection{Absence of diffusion-driven (Turing) instability in the base reaction--diffusion subsystem}
\label{subsec:main_noturing}

We show that the coexistence equilibrium of the limiting kinetics \eqref{eq:SR_reduced} cannot be destabilized by diffusion
in the base reaction--diffusion subsystem.
In contrast to classical activator--inhibitor kinetics, the Jacobian at the coexistence equilibrium has strictly negative
diagonal entries, so diffusion only increases damping of nonconstant modes.

\begin{theorem}[No diffusion-driven (Turing) instability at the coexistence equilibrium]
\label{thm:no_turing}
Assume $K>0$ and $\alpha+\xi>0$ so that \eqref{eq:SR_reduced} admits the coexistence equilibrium
$\mathbb{K}^*=(S^*,R^*)$ given by \eqref{eq:Kstar_reduced}.
Consider the reaction--diffusion system
\begin{equation}
\label{eq:RD_SR}
\partial_t S = d_S\Delta S + f_S(S,R),\qquad
\partial_t R = d_R\Delta R + f_R(S,R),
\end{equation}
posed on $U$ with homogeneous Neumann boundary conditions.
Then, for any diffusion coefficients $d_S,d_R>0$, the homogeneous steady state
\[
(S(x,t),R(x,t))\equiv (S^*,R^*)
\]
is linearly asymptotically stable with respect to all nonconstant Neumann Laplacian modes (i.e.\ all modes $k\ge2$ with
$\lambda_k>0$). In particular, no diffusion-driven (Turing) instability occurs at $\mathbb{K}^*$.
\end{theorem}

\begin{remark}[Sign mechanism in the dispersion relation]
\label{rem:dispersion_mechanism_noturing}
In classical two-species Turing systems, diffusion can destabilize a linearly stable equilibrium because mixed signs in the Jacobian
interact with unequal diffusion and can force the constant term in the dispersion polynomial to become negative for some $\lambda_k>0$.
Here, the strict negativity of the diagonal derivatives \eqref{eq:diag_signs_noturing} implies that diffusion only increases damping:
the $\lambda_k$-coefficient in $a_0(\lambda_k)$ is nonnegative and the quadratic \eqref{eq:a1a0_noturing} remains positive for all $\lambda_k>0$. 
\end{remark}

\subsection{Directionality principle for chemotaxis with a diffusive signal}
\label{subsec:main_directionality}

We now incorporate chemotaxis via the diffusive chemoattractant $c(x,t)$ introduced in Section~\ref{subsec:chemotaxis_extensions_overview}.
Our main message is a directionality principle: in the one-way damped regime, chemotactic coupling does not alter the
population spectrum and patterns are suppressed; in contrast, two-way feedback can induce effective cross-diffusion and recover
Turing-type instabilities. 
Our mode-wise stability analysis follows the classical spectral reduction under Neumann eigenmodes
and yields explicit algebraic criteria (trace/determinant or Routh--Hurwitz type).
For complementary developments on explicit algebraic--spectral thresholds and their robustness under
time discretization, see \citep{wang_algebraicspectral_2026}.
Throughout we work under homogeneous Neumann boundary conditions. Figure~\ref{fig:dispersion} distinguishes the dispersion curves under the unidirectional and bidirectional feedback loop. The bidirectional feedback contains the self-reinforcing ingredients required for pattern formation, manifesting as positive eigenvalues in the dispersion curve.

\begin{figure}[hbtp]
    \centering
    \includegraphics[width=0.75\linewidth]{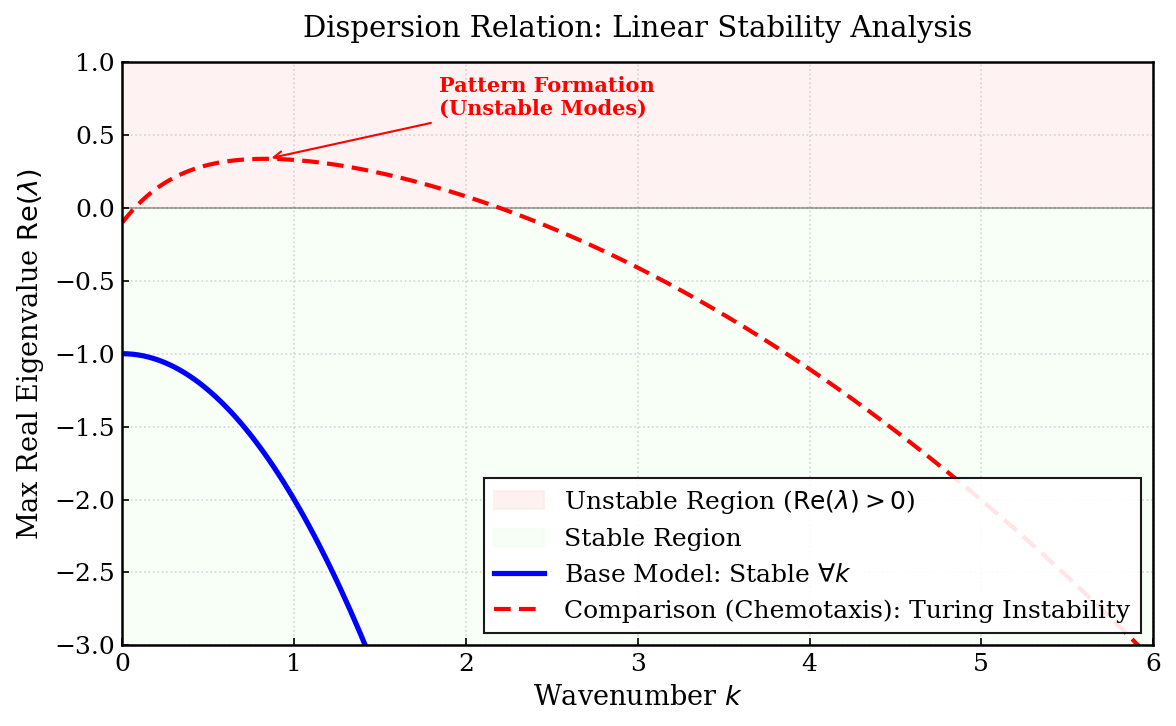}
    \caption{\textbf{Comparison of dispersion relations under base model and the chemotaxis extended model.} In the base model, the tumor population $(S,R)$ responds to the signal but do not regulate it, forming a unidirectional feedback loop where all spatial perturbations decay under damping structures. In the extended model, the signal is generated by the population, closing the bidirectional feedback loop required for pattern formation. The pattern formation requires the maximal eigenvalue to be positive in the dispersion relation.}
    \label{fig:dispersion}
\end{figure}

\subsubsection{One-way damped coupling: pattern suppression}
\label{subsubsec:main_oneway}

Consider the one-way damped chemotaxis model
\begin{equation}
\label{eq:oneway_damped_c_main}
\begin{cases}
\partial_t S = d_S\Delta S - \chi_S\nabla\cdot(S\nabla c) + f_S(S,R),\\[2pt]
\partial_t R = d_R\Delta R - \chi_R\nabla\cdot(R\nabla c) + f_R(S,R),\\[2pt]
\partial_t c = d_c\Delta c + \mathcal{Q}(c),
\end{cases}
\qquad \partial_{\mathrm n}S=\partial_{\mathrm n}R=\partial_{\mathrm n}c=0,
\end{equation}
where $\mathcal{Q}$ is $C^1$ and $\mathcal{Q}'(c^*)<0$ at the relevant homogeneous equilibrium $c^*$.
(As discussed in Section~\ref{subsec:chemotaxis_extensions_overview}, this template also covers the hybrid forcing
$\partial_t c=d_c\Delta c+\kappa A-\rho c$ since $A$ evolves independently of $(S,R)$ in the one-way setting and the $c$-equation
is linearly damped.)

\begin{theorem}[One-way damped chemotaxis suppresses mode-wise pattern formation]
\label{thm:oneway_suppression}
Assume that \eqref{eq:oneway_damped_c_main} admits a spatially homogeneous equilibrium
$\mathbf{u}^*=(S^*,R^*,c^*)$ with $S^*,R^*>0$ and $\mathcal{Q}'(c^*)<0$.
Let $\mathcal{J}_F\coloneqq D(f_S,f_R)(S^*,R^*)$ and assume kinetic stability:
\begin{equation}
\label{eq:kinetic_stability_oneway}
\operatorname{tr}(\mathcal{J}_F)<0,\qquad \det(\mathcal{J}_F)>0.
\end{equation}
If, moreover, the corresponding base reaction--diffusion subsystem is stable for all nonconstant Neumann modes, i.e.
\begin{equation}
\label{eq:RD_block_stable_assumption}
\operatorname{Spec}\bigl(\mathcal{J}_F-\lambda_k \mathrm{diag}(d_S,d_R)\bigr)\subset\{\Re z<0\}
\quad \text{for all }k\ge2,
\end{equation}
then the full one-way system \eqref{eq:oneway_damped_c_main} is linearly asymptotically stable with respect to all nonconstant modes.
In particular, no diffusion/chemotaxis-driven (Turing-type) instability occurs in the one-way damped setting, for any
$\chi_S,\chi_R\ge0$ and $d_c>0$.
\end{theorem}

\subsubsection{Two-way feedback: recovery of instability via explicit mode-wise criteria}
\label{subsubsec:main_twoway}

To capture feedback loops characteristic of chemotaxis-driven patterning, we consider the reduced (fast-relaxation) closure
in which the diffusive signal satisfies $c\approx g(S,R)$ locally, yielding the two-species cross-diffusion system
\begin{equation}
\label{eq:twoway_reduced_g}
\begin{cases}
\partial_t S = d_S\Delta S - \chi_S\nabla\cdot\bigl(S\nabla g(S,R)\bigr) + f_S(S,R),\\[3pt]
\partial_t R = d_R\Delta R - \chi_R\nabla\cdot\bigl(R\nabla g(S,R)\bigr) + f_R(S,R),
\end{cases}
\qquad \partial_{\mathrm n}S=\partial_{\mathrm n}R=0,
\end{equation}
where $g\in C^1$.
(Section~\ref{subsec:chemotaxis_extensions_overview} describes how \eqref{eq:twoway_reduced_g} arises from a PDE signal model
via fast relaxation; a mode-wise elimination for the full PDE system is placed Appendix~\ref{app:closure}.)

\begin{theorem}[Two-way feedback recovers Turing-type instability: trace/determinant criteria]
\label{thm:twoway_criteria}
Assume \eqref{eq:twoway_reduced_g} admits a spatially homogeneous interior equilibrium $\mathbf{u}^*=(S^*,R^*)$.
Let
\[
\mathcal{J}_F \coloneqq D(f_S,f_R)(S^*,R^*)=
\begin{pmatrix}
a & b\\
c & d
\end{pmatrix},
\qquad
g_S^*\coloneqq \partial_S g(S^*,R^*),\quad g_R^*\coloneqq \partial_R g(S^*,R^*),
\]
and assume kinetic stability:
\begin{equation}
\label{eq:kinetic_stability_twoway}
\operatorname{tr}(\mathcal{J}_F)<0,\qquad \det(\mathcal{J}_F)>0.
\end{equation}
For each Neumann mode $k\ge2$ with eigenvalue $\lambda_k>0$, the mode matrix is
\[
\mathcal{A}(\lambda_k)=\mathcal{J}_F-\lambda_k(\mathcal{D}-\mathcal{H}),
\qquad
\mathcal{D}\coloneqq \mathrm{diag}(d_S,d_R),
\]
where
\[
\mathcal{H}\coloneqq
\begin{pmatrix}
\chi_S g_S^* S^* & \chi_S g_R^* S^*\\[4pt]
\chi_R g_S^* R^* & \chi_R g_R^* R^*
\end{pmatrix}.
\]
Then the equilibrium $\mathbf{u}^*$ is unstable for some nonconstant mode (i.e.\ a Turing-type instability occurs) if and only if
there exists $k\ge2$ such that at least one of the following holds:
\begin{enumerate}[label=(\alph*), leftmargin=*]
\item \textbf{Trace condition:}
\begin{equation}
\label{eq:twoway_trace_condition}
\operatorname{tr}(\mathcal{A}(\lambda_k))
=
\operatorname{tr}(\mathcal{J}_F)
-\lambda_k(d_S+d_R)
+\lambda_k\bigl(\chi_S g_S^* S^*+\chi_R g_R^* R^*\bigr)
>0.
\end{equation}

\item \textbf{Determinant condition:}
$\det(\mathcal{A}(\lambda_k))<0$ for some $k\ge2$.
Equivalently, writing
\begin{equation}
\label{eq:det_quadratic_def}
\det(\mathcal{A}(\lambda))=A_1\lambda^2+A_2\lambda+\det(\mathcal{J}_F),
\qquad \lambda\ge0,
\end{equation}
with
\begin{align}
\label{eq:A1A2_def}
A_{1}
&= d_{S}d_{R}-d_{S}\chi_{R}g_R^*R^*-d_{R}\chi_{S}g_S^*S^*,\nonumber\\
A_{2}
&= -ad_{R}-dd_{S}
+a\chi_{R}g_R^*R^* + d\chi_{S}g_S^*S^*
-b\chi_{R}g_S^*R^* -c\chi_{S}g_R^*S^*,
\end{align}
there exists $k\ge2$ with $\det(\mathcal{A}(\lambda_k))<0$ if and only if one of the following (mutually exclusive) cases holds:
\begin{enumerate}[label=(\roman*), leftmargin=*]
\item $A_1<0$ (then $\det(\mathcal{A}(\lambda))\to -\infty$ as $\lambda\to\infty$ and large modes destabilize);
\item $A_1=0$ and $A_2<0$ (then $\det(\mathcal{A}(\lambda))=A_2\lambda+\det(\mathcal{J}_F)$ becomes negative for sufficiently large $\lambda$);
\item $A_1>0$, $A_2<0$, and $A_2^2-4A_1\det(\mathcal{J}_F)>0$ (then the upward-opening quadratic attains negative values between its two positive roots).
\end{enumerate}
\end{enumerate}
\end{theorem}

\begin{corollary}[Simple sufficient parameter regimes for determinant-driven instability]
\label{cor:twoway_simple_regimes}
Maintain the notation of Theorem~\ref{thm:twoway_criteria}.
Any one of the following conditions is sufficient to guarantee the existence of a nonconstant mode $k\ge2$ with
$\det(\mathcal{A}(\lambda_k))<0$ (hence a Turing-type instability under \eqref{eq:kinetic_stability_twoway}):
\begin{enumerate}[label=(S\arabic*), leftmargin=*]
\item \textbf{Single-species strong feedback:}\quad
$\chi_S g_S^* S^* > d_S$ \ \textbf{or}\ \ $\chi_R g_R^* R^* > d_R$ \quad (implies $A_1<0$).

\item \textbf{Combined feedback dominates diffusion:}\quad
$\dfrac{\chi_S g_S^* S^*}{d_S}+\dfrac{\chi_R g_R^* R^*}{d_R}>1$ \quad (equivalent to $A_1<0$).

\item \textbf{Degenerate (affine) route:}\quad
$A_1=0$ and $A_2<0$.

\item \textbf{Upward-parabola deep dip:}\quad
$A_1>0$, $A_2<0$, and $A_2^{2} > 4A_1\det(\mathcal{J}_F)$.
\end{enumerate}
\end{corollary}

\begin{remark}[Connection to the full PDE signal model]
\label{rem:twoway_pde_connection}
The reduced closure \eqref{eq:twoway_reduced_g} is motivated by fast signal relaxation in the PDE model
$\partial_t c=d_c\Delta c+q\,c+h(S,R)$ with $q<0$.
For each Neumann mode, one may eliminate the signal amplitude by a Schur complement argument, obtaining an effective
$2\times2$ mode matrix for $(S,R)$ whose coefficients depend on $\lambda_k$.
Theorem~\ref{thm:twoway_criteria} provides explicit, closed-form criteria in the reduced local closure $c\approx g(S,R)$.
\end{remark}

\section{Proofs}
\label{sec:analysis}

\subsection{Proof of local Lipschitz continuity of the reaction map (Lemma~\ref{lem:local_lipschitz})}
\label{subsec:analysis_Lipschitz}

Recall the nonlinear map $\mathcal{G}: \mathbb{R}^5\to\mathbb{R}^5$ defined by the right-hand side of \eqref{eq:full_system} and
the structural assumptions \eqref{eq:delta_def}--\eqref{eq:phi_def}.

\begin{proof}
Fix $M>0$ and let $\mathbf{z}=(S,R,D,P,A)$ and $\mathbf{y}=(\tilde S,\tilde R,\tilde D,\tilde P,\tilde A)$ belong to $[-M,M]^5$.
Each component of $\mathcal{G}$ is a finite sum of products of bounded variables with Lipschitz nonlinearities in $D$.
Specifically, $\delta$ is $C^1$ on $[0,\infty)$ with bounded derivative on $[0,M]$, hence Lipschitz there, and $\phi$ is globally Lipschitz by
\eqref{eq:phi_def}. The remaining terms are polynomial (at most bilinear) in $(S,R,P,A)$ and therefore Lipschitz on bounded sets.
Collecting the estimates yields the stated bound with a constant $L_M$ depending only on $M$ and the model parameters.
\end{proof}

\subsection{Proof of well-posedness and positivity}
\label{subsec:analysis_wellposed}

This section proves Theorems~\ref{thm:positivity} and~\ref{thm:global_wellposed}.
We work in the Banach space $\mathbf{X}= \bigl(C(\overline U)\bigr)^3 \times \bigl(L^\infty(U)\bigr)^2$ with positive cone $\mathbf{X}_+$, and use the mild formulation
\eqref{eq:mild_formula_shifted} generated by the positivity-preserving semigroup $\mathbb{P}(t)$ defined in \eqref{eq:semigroup_def_shifted}.
Throughout, the structural hypotheses \eqref{eq:delta_def}--\eqref{eq:phi_def} and the homogeneous Neumann boundary conditions
\eqref{eq:BC_core} are in force.

\subsubsection{Local existence and uniqueness.}
By Lemma~\ref{lem:local_lipschitz}, the nonlinear map $\mathcal{G}:\mathbf{X}\to\mathbf{X}$ induced by the reaction terms
is locally Lipschitz in $\mathbf{X}$.
Standard semilinear evolution theory in Banach spaces therefore yields the existence of a unique maximal mild solution
$\mathbf{u}\in C([0,T_{\max});\mathbf{X})$ satisfying \eqref{eq:mild_formula_shifted}, where either $T_{\max}=\infty$ or
\begin{equation}
\label{eq:blowup_alternative_analysis}
T_{\max}<\infty\quad \Longrightarrow\quad
\limsup_{t\uparrow T_{\max}}\|\mathbf{u}(t)\|_{\mathbf{X}}=+\infty.
\end{equation}

\subsubsection{Proof of Theorem~\ref{thm:positivity}.}
\begin{proof}
We first verify the quasi-positivity condition on $\mathcal{G}$: for each component $i\in\{1,\dots,5\}$,
\begin{equation}
\label{eq:quasi_positivity_def}
z_i=0,\ z_j\ge0\ (j\neq i)\quad \Longrightarrow\quad G_i(\mathbf{z})\ge0.
\end{equation}
Using $\delta(D)\ge 0$, $\phi(D)\in[0,1)$ and $1-\phi(D)\ge 0$ for $D\ge 0$, we have for $\mathbf u\ge 0$,
\[
G_1\big|_{S=0}=\xi\bigl[1-\phi(D)\bigr]R\ge 0,\quad
G_2\big|_{R=0}=\alpha S\ge 0,\quad
G_3\big|_{D=0}=0,
\]
\[
G_4\big|_{P=0}=\beta\bigl[1-\phi(D)\bigr]A\ge 0,\quad
G_5\big|_{A=0}=\theta\phi(D)P\ge 0.
\]
Thus $\mathcal G$ is quasi-positive on $\mathbb R_{\ge 0}^5$.

We show positivity of the diffusive component $S$.
Fix $T\in(0,T_{\max})$.
Since $\mathbf u$ is classical on $[0,T]$, we have
\[
M_T \coloneqq \sup_{t\in[0,T]}\|\mathbf u(t)\|_{L^\infty(U)} < \infty.
\]
Write the $S$-equation in the form
\[
\partial_t S - d_S\Delta S = f_S(S,R,D,P,A),
\]
where
\[
f_S(S,R,D,P,A)
=\lambda_S S\Bigl(1-\tfrac{S+R}{K}\Bigr)-\alpha S-\delta(D)S+\xi\bigl[1-\phi(D)\bigr]R.
\]
For a.e.\ $(x,t)$, define the negative part $S_-(x,t)\coloneqq \max\{-S(x,t),0\}$.
Testing the $S$-equation with $S_-$ and integrating over $U$ yields, for a.e.\ $t\in(0,T)$,
\begin{align*}
\frac12\frac{d}{dt}\|S_-(t)\|_{L^2(U)}^2
&= -d_S\int_U |\nabla S_-(x,t)|^2\,dx - \int_U f_S(S,R,D,P,A)\,S_-(x,t)\,dx.
\end{align*}
(The diffusion term is handled by integration by parts; the boundary term vanishes due to $\partial_{\mathrm n}S=0$.)

Since $f_S$ is locally Lipschitz in $S$ on bounded sets by Lemma~\ref{lem:local_lipschitz}, there exists $L_T>0$
(depending only on $M_T$ and parameters) such that
\[
|f_S(S,R,D,P,A)-f_S(0,R,D,P,A)|\le L_T |S|
\qquad\text{whenever }|S|\le M_T,\ 0\le R,D,P,A\le M_T.
\]
On the set $\{x\in U:\ S(x,t)<0\}$ we have $|S|=S_-$, hence
\[
f_S(S,R,D,P,A)\ge f_S(0,R,D,P,A)-L_T S_-.
\]
Moreover, for all $(R,D,P,A)\ge0$ we have $f_S(0,R,D,P,A)=\xi[1-\phi(D)]R\ge0$.
Therefore,
\[
f_S(S,R,D,P,A)\ge -L_T S_-
\qquad\text{a.e.\ on the set }\{S<0\}.
\]
Consequently,
\[
-\int_U f_S(S,R,D,P,A)\,S_-\,dx
= -\int_{\{S<0\}} f_S(S,R,D,P,A)\,S_-\,dx
\le L_T\int_U S_-^2\,dx  = L_T\|S_-(t)\|_{L^2(U)}^2,
\]
and therefore
\[
\frac{d}{dt}\|S_-(t)\|_{L^2(U)}^2 \le 2L_T\,\|S_-(t)\|_{L^2(U)}^2
\qquad\text{for a.e.\ }t\in(0,T).
\]
Since $S_-(0)=0$ by $S_0\ge 0$, Gr\"onwall's inequality gives $\|S_-(t)\|_{L^2(U)}^2\equiv 0$ on $[0,T]$.
Thus $S(x,t)\ge 0$ for all $(x,t)\in \overline U\times[0,T]$.
As $T<T_{\max}$ was arbitrary, we conclude $S\ge 0$ on $\overline U\times[0,T_{\max})$.

Applying the same argument to $R$ gives $R\ge 0$.
For $D$ we have $\partial_tD-d_D\Delta D=-\gamma_d D$, and the same test with $D_-$ yields $D\ge 0$.

Fix $x\in U$. The variables $(P(x,t),A(x,t))$ satisfy the linear (time-dependent) system
\[
\frac{d}{dt}\begin{pmatrix}
    P(x,t)\\[4pt]
    A(x,t)
\end{pmatrix}
=
\begin{pmatrix}
-\theta\phi(D(x,t)) & \beta\bigl[1-\phi(D(x,t))\bigr]\\[4pt]
\theta\phi(D(x,t)) & -\beta\bigl[1-\phi(D(x,t))\bigr]
\end{pmatrix}
\begin{pmatrix}
    P(x,t)\\[4pt]
    A(x,t)
\end{pmatrix}.
\]
The off-diagonal entries are nonnegative for $D\ge 0$, hence the nonnegative quadrant is forward invariant:
if $P(x,t_0)=0$ and $A(x,t_0)\ge 0$, then
$P_t(x,t_0)=\beta[1-\phi(D(x,t_0))]A(x,t_0)\ge 0$; similarly, if $A(x,t_0)=0$ and $P(x,t_0)\ge 0$, then
$A_t(x,t_0)=\theta\phi(D(x,t_0))P(x,t_0)\ge 0$.
Therefore $P(x,t),A(x,t)\ge 0$ for all $t\in[0,T_{\max})$.

This completes the proof of Theorem~\ref{thm:positivity}.
\end{proof}

\subsubsection{A priori bounds used for continuation and proof of Theorem~\ref{thm:global_wellposed}.}
\begin{proof}
We next derive the bounds needed to rule out \eqref{eq:blowup_alternative_analysis}.
First, the signal equation for $D$ is decoupled and linear, so Proposition~\ref{prop:D_decay} gives
\begin{equation}
\label{eq:analysis_D_bound}
0\le D(x,t)\le \|D_0\|_{L^\infty(U)}\qquad \forall x\in U,\ \forall t\ge0.
\end{equation}
Second, Lemma~\ref{lem:PA_conservation} gives the uniform-in-time estimate
\begin{equation}
\label{eq:analysis_PA_bound}
\sup_{t\ge0}\bigl(\|P(t)\|_{L^\infty(U)}+\|A(t)\|_{L^\infty(U)}\bigr)
\le \|P_0\|_{L^\infty(U)}+\|A_0\|_{L^\infty(U)}.
\end{equation}

Finally, fix $T\in(0,T_{\max})$ and $p\ge2$. Define $W_p(t)\coloneqq \|S(t)\|_{L^p}^p+\|R(t)\|_{L^p}^p$,
testing the $S$- and $R$-equations in \eqref{eq:full_system} against $pS^{p-1}$ and $pR^{p-1}$, using $S,R\ge0$,
dropping the nonpositive diffusion terms, and applying Young's inequality to the coupling terms $SR^{p-1}$ and $S^{p-1}R$ yields
\begin{equation}
\label{eq:analysis_Wp_gronwall}
\frac{d}{dt}W_p(t)\le C_p\,W_p(t)\qquad \text{for }t\in[0,T],
\end{equation}
where $C_p$ depends linearly on $p$, the parameters, and $\|A\|_{L^\infty(U\times(0,T))}$.

By \eqref{eq:analysis_PA_bound}, $\|A\|_{L^\infty(U\times(0,T))}$ is bounded by a constant independent of $t$,
hence Gr\"onwall's inequality implies
\begin{equation}
\label{eq:analysis_Wp_bound}
W_p(t)\le W_p(0)\,e^{C_p t}\qquad \forall t\in[0,T].
\end{equation}
Letting $p\to\infty$ yields an $L^\infty$ bound for $S$ and $R$ on $[0,T]$:
\begin{equation}
\label{eq:analysis_SR_Linfty}
\sup_{0\le t\le T}\bigl(\|S(t)\|_{L^\infty(U)}+\|R(t)\|_{L^\infty(U)}\bigr)\le C_T,
\end{equation}
for some $C_T<\infty$.

Combining \eqref{eq:analysis_D_bound}, \eqref{eq:analysis_PA_bound}, and \eqref{eq:analysis_SR_Linfty} shows that
$\sup_{0\le t\le T}\|\mathbf{u}(t)\|_{\mathbf{X}}<\infty$ for every $T<T_{\max}$.
This precludes the blow-up alternative \eqref{eq:blowup_alternative_analysis}, hence $T_{\max}=\infty$ and the mild solution is global.
Uniqueness follows from the standard contraction argument on bounded time intervals using local Lipschitz continuity of $\mathcal{G}$.
\end{proof}

\subsection{Proof of the long-time reduction and attractivity of the limiting kinetics}
\label{subsec:analysis_longtime}

We prove Theorems~\ref{thm:longtime_reduction} and~\ref{thm:global_attract_reduced}.
Throughout, $\mathbf{u}=(S,R,D,P,A)$ denotes the global mild solution of \eqref{eq:full_system}
given by Theorem~\ref{thm:global_wellposed}, with $\mathbf{u}_0\in\mathbf{X}_+$.

\subsubsection*{Proof of Theorem~\ref{thm:longtime_reduction}}
\begin{proof}
By Proposition~\ref{prop:D_decay} we have the energy decay
\[
\|D(t)\|_{L^2(U)}\le e^{-\gamma_d t}\|D_0\|_{L^2(U)}\qquad \forall t\ge0.
\]
Moreover, by the maximum principle (again Proposition~\ref{prop:D_decay}), $0\le D(x,t)\le \|D_0\|_{L^\infty(U)}$ for all $t\ge0$.
Since $\phi$ is globally Lipschitz on $[0,\infty)$ with $\phi(0)=0$ and $\delta$ is Lipschitz on $[0,\|D_0\|_{L^\infty}]$ with $\delta(0)=0$,
we obtain
\[
\|\phi(D(t))\|_{L^2(U)}\le L_\phi \|D(t)\|_{L^2(U)}\to0,
\qquad
\|\delta(D(t))\|_{L^2(U)}\le L_\delta \|D(t)\|_{L^2(U)}\to0,
\]
as $t\to\infty$, where $L_\phi$ is the (global) Lipschitz constant of $\phi$ and $L_\delta$ is a Lipschitz constant of $\delta$
on $[0,\|D_0\|_{L^\infty}]$.

Fix $T>0$. For $t\ge T$, using the uniform-in-time bounds on $S,R$ over compact time intervals
(Theorem~\ref{thm:global_wellposed}) and the uniform bound on $A$ from \eqref{eq:PA_Linfty_bound}, we estimate
\begin{align*}
\|\delta(D(t))S(t)\|_{L^\infty(U)} &\le \|\delta(D(t))\|_{L^\infty(U)}\,\|S(t)\|_{L^\infty(U)}, \\[6pt]
\|\phi(D(t))A(t)R(t)\|_{L^\infty(U)} &\le \|\phi(D(t))\|_{L^\infty(U)}\,\|A(t)\|_{L^\infty(U)}\,\|R(t)\|_{L^\infty(U)}.
\end{align*}
Since $0\le D(\cdot,t)\le\|D_0\|_{L^\infty(U)}$ and $D(\cdot,t)\to0$ in $L^2(U)$, we have (after extracting a subsequence if needed)
$D(\cdot,t)\to0$ a.e.\ as $t\to\infty$, and by continuity of $\phi$ and $\delta$ at $0$,
$\phi(D(\cdot,t))\to0$ and $\delta(D(\cdot,t))\to0$ a.e.
Combined with boundedness and dominated convergence, this yields
\[
\|\phi(D(t))\|_{L^p(U)}\to0,\qquad \|\delta(D(t))\|_{L^p(U)}\to0
\quad \text{for every finite }p\ge1.
\]
In particular, along any sequence $t_n\to\infty$ for which $D(\cdot,t_n)\to0$ a.e., we have
$\|\phi(D(t_n))\|_{L^\infty(U)}\to0$ and $\|\delta(D(t_n))\|_{L^\infty(U)}\to0$ provided $D(\cdot,t_n)\to0$ uniformly.
(Uniform decay holds, for example, when $D_0\in C(\overline U)$ by standard parabolic smoothing; see Appendix~\ref{app:spectral}.)
Hence all $D$-mediated couplings in the $(S,R)$-equations vanish asymptotically, and the limiting kinetics are given by \eqref{eq:SR_reduced}.

This proves Theorem~\ref{thm:longtime_reduction}.
\end{proof}

\begin{remark}[On uniform decay of $D$]
If one assumes $D_0\in C(\overline U)$ (or higher regularity), then the solution $D(\cdot,t)$ of the damped heat equation is continuous for
all $t>0$ and decays to $0$ uniformly on $\overline U$ as $t\to\infty$.
In that case, $\|\phi(D(t))\|_{L^\infty}$ and $\|\delta(D(t))\|_{L^\infty}$ also converge to $0$ as $t\to\infty$, and the estimates in Step~2
hold without subsequences.
\end{remark}

\subsubsection*{Proof of Theorem~\ref{thm:global_attract_reduced}}
\begin{proof}
Let $(S(t),R(t))$ solve \eqref{eq:SR_reduced} with $S(0),R(0)\ge0$ and $S(0)+R(0)>0$.
Positively invariance of the nonnegative quadrant follows since $\dot S=\xi R\ge0$ on $\{S=0\}$ and $\dot R=\alpha S\ge0$ on $\{R=0\}$.

Let $N(t)\coloneqq S(t)+R(t)$. Summing \eqref{eq:SR_reduced} gives
\[
\dot N = \Bigl(1-\frac{N}{K}\Bigr)\bigl(\lambda_S S+\lambda_R R\bigr).
\]
Since $\lambda_S S+\lambda_R R\ge0$ and $N(t)>0$, we have $\dot N>0$ for $0<N<K$ and $\dot N<0$ for $N>K$.
Hence $N(t)$ is monotone toward $K$ and bounded, so
\begin{equation}
\label{eq:analysis_N_to_K}
N(t)\to K \qquad \text{as } t\to\infty.
\end{equation}

Let $S^*=\frac{\xi K}{\alpha+\xi}$ and define $E(t)\coloneqq S(t)-S^*$.
Using $R(t)=N(t)-S(t)$ and $(\alpha+\xi)S^*=\xi K$, we compute
\begin{align}
\dot E
&= \dot S
= \lambda_S S\Bigl(1-\frac{N}{K}\Bigr)-\alpha S+\xi(N-S)\nonumber\\
&= -(\alpha+\xi)(S-S^*) + \lambda_S S\Bigl(1-\frac{N}{K}\Bigr) + \xi(N-K)\nonumber\\
&= -(\alpha+\xi)E + \underbrace{\lambda_S S\Bigl(1-\frac{N}{K}\Bigr)+\xi(N-K)}_{=:F(t)}.
\label{eq:E_forced}
\end{align}
By \eqref{eq:analysis_N_to_K} and boundedness of $S(t)$, we have $F(t)\to0$ as $t\to\infty$.
Applying the variation-of-constants formula to \eqref{eq:E_forced} yields
\[
E(t)=e^{-(\alpha+\xi)t}E(0)+\int_0^t e^{-(\alpha+\xi)(t-s)}F(s)\,ds,\qquad t\ge 0.
\]
Since $F(t)\in L^\infty(0,\infty)$,
\begin{enumerate}[label=(\roman*)]
\item $E(t)\to0$ as $t\to\infty$.
\item If there exist $T\ge0$ and $\varepsilon>0$ with $|F(t)|\le\varepsilon$ for all $t\ge T$, then for $t\ge T$
\[
|E(t)| \le e^{-\lambda t}|E(0)| + e^{-\lambda t}\int_0^T e^{\lambda s}|F(s)|\,ds + \frac{\varepsilon}{\lambda},
\]
hence $\limsup_{t\to\infty}|E(t)|\le\varepsilon/\lambda$.
\item If $|F(t)|\le C e^{-\beta t}$ for constants $C,\beta>0$, then $E(t)=O(e^{-\gamma t})$ with $\gamma=\min\{\lambda,\beta\}$ (explicit formulas above).
\end{enumerate}
Consequently, $E(t)\to0$ as $t\to\infty$, i.e.\ $S(t)\to S^*$.
Finally, \eqref{eq:analysis_N_to_K} and $R=N-S$ imply $R(t)\to K-S^*=\frac{\alpha K}{\alpha+\xi}=R^*$.
This proves global attractivity of $\mathbb{K}^*=(S^*,R^*)$.
\end{proof}

\subsection{Mode-wise stability proofs}
\label{subsec:analysis_modewise}

We collect the mode-wise stability arguments underlying Theorems~\ref{thm:no_turing}, \ref{thm:oneway_suppression},
and~\ref{thm:twoway_criteria}. Throughout, let $\{(\lambda_k,\omega_k)\}_{k\ge1}$ denote the Neumann Laplacian eigenpairs on $U$:
\[
-\Delta \omega_k=\lambda_k\omega_k\ \text{in }U,\qquad \partial_{\mathrm n}\omega_k=0\ \text{on }\partial U,
\qquad 0=\lambda_1<\lambda_2\le\cdots,\ \lambda_k\to\infty,
\]
with $\{\omega_k\}$ orthonormal in $L^2(U)$.

\subsubsection*{Proof of Theorem~\ref{thm:no_turing} (no diffusion-driven instability)}
\begin{proof}
Let $F(S,R)\coloneqq (f_S(S,R),f_R(S,R))^{\mathsf T}$ and denote by
\[
\mathcal{J}\coloneqq DF(\mathbb{K}^*)=
\begin{pmatrix}
a & b\\
c & d
\end{pmatrix}
\]
the Jacobian of $F$ at $\mathbb{K}^*$.
From \eqref{eq:SR_reduced} one computes
\begin{align}
\partial_S f_S(S,R) &= \lambda_S\Bigl(1-\frac{2S+R}{K}\Bigr)-\alpha,\qquad
\partial_R f_S(S,R)= -\frac{\lambda_S}{K}S+\xi,\nonumber\\
\partial_S f_R(S,R) &= -\frac{\lambda_R}{K}R+\alpha,\qquad
\partial_R f_R(S,R)= \lambda_R\Bigl(1-\frac{S+2R}{K}\Bigr)-\xi.
\label{eq:partials_SR_noturing}
\end{align}
At $\mathbb{K}^*$ we have $S^*=\dfrac{\xi K}{\alpha+\xi}$ and $R^*=\dfrac{\alpha K}{\alpha+\xi}$, hence
\begin{equation}
\label{eq:diag_signs_noturing}
a=\partial_S f_S(\mathbb{K}^*)=-\frac{\xi\lambda_S}{\alpha+\xi}-\alpha<0,
\qquad
d=\partial_R f_R(\mathbb{K}^*)=-\frac{\alpha\lambda_R}{\alpha+\xi}-\xi<0.
\end{equation}
Moreover,
\begin{equation}
\label{eq:tr_det_noturing}
\operatorname{tr}(\mathcal{J})=a+d<0,
\qquad
\det(\mathcal{J})=\xi\lambda_S+\alpha\lambda_R>0,
\end{equation}
so the homogeneous equilibrium is linearly stable.

Seeking normal modes $(s,r)=(\hat s,\hat r)\omega_k(x)e^{\lambda t}$ in the linearization of \eqref{eq:RD_SR} at $\mathbb{K}^*$
yields, for each $k\ge1$, the mode matrix
\[
\mathcal{A}(\lambda_k)\coloneqq \mathcal{J}-\lambda_k D_{SR},
\qquad
D_{SR}\coloneqq \mathrm{diag}(d_S,d_R),
\]
and the dispersion relation $\det(\lambda I-\mathcal{A}(\lambda_k))=0$.
Equivalently, $\lambda$ satisfies the quadratic
\begin{equation}
\label{eq:dispersion_poly_noturing}
\lambda^2 + a_1(\lambda_k)\lambda + a_0(\lambda_k)=0,
\end{equation}
where
\begin{align}
\label{eq:a1a0_noturing}
a_1(\lambda_k)
&=\lambda_k(d_S+d_R)-\operatorname{tr}(\mathcal{J}),
\nonumber\\[4pt]
a_0(\lambda_k)
&=d_Sd_R\lambda_k^2-\bigl(d_S d + d_R a\bigr)\lambda_k+\det(\mathcal{J}).
\end{align}
For any nonconstant mode $k\ge2$ we have $\lambda_k>0$, hence
\[
a_1(\lambda_k)=\lambda_k(d_S+d_R)-\operatorname{tr}(\mathcal{J})>\lambda_k(d_S+d_R)\ge0,
\]
so $a_1(\lambda_k)>0$.
Using $a<0$ and $d<0$ from \eqref{eq:diag_signs_noturing}, we also have $-(d_S d+d_R a)\lambda_k>0$ for $\lambda_k>0$, and therefore
\[
a_0(\lambda_k)\ge d_Sd_R\lambda_k^2+\det(\mathcal{J})>0
\qquad \text{for all }k\ge2.
\]
By the Routh--Hurwitz criterion for quadratic polynomials, $a_1(\lambda_k)>0$ and $a_0(\lambda_k)>0$ imply
$\Re(\lambda)<0$ for both roots of \eqref{eq:dispersion_poly_noturing}. Hence every nonconstant mode is exponentially damped.

Finally, the constant mode $k=1$ corresponds to $\lambda_1=0$ and is stable by \eqref{eq:tr_det_noturing}.
Therefore $(S^*,R^*)$ is linearly asymptotically stable with respect to all Neumann modes, and no diffusion-driven
(Turing) instability can occur.
\end{proof}

\subsubsection*{Proof of Theorem~\ref{thm:oneway_suppression} (one-way damped chemotaxis)}
\begin{proof}
Let $\mathbf{u}^*=(S^*,R^*,c^*)$ be a homogeneous equilibrium of the one-way system \eqref{eq:oneway_damped_c_main}.
Write $\mathcal{J}_F=D(f_S,f_R)(S^*,R^*)$ and denote $D_{SR}=\mathrm{diag}(d_S,d_R)$.

Linearizing \eqref{eq:oneway_damped_c_main} at $\mathbf{u}^*$ and projecting onto an eigenmode $\omega_k$ with $\lambda_k>0$
(i.e.\ $k\ge2$) yields a $3\times3$ linear ODE for the mode amplitudes
$\hat{\mathbf{z}}=(\hat s,\hat r,\hat c)^{\mathsf T}$:
\begin{equation}
\label{eq:analysis_oneway_mode}
\frac{d}{dt}\hat{\mathbf{z}}=\mathcal{M}(\lambda_k)\hat{\mathbf{z}},
\qquad
\mathcal{M}(\lambda_k)=
\begin{pmatrix}
\mathcal{J}_F-\lambda_k D_{SR} & \lambda_k\,\mathbf{b}\\[2pt]
\mathbf{0}^{\mathsf T} & \mathcal{Q}'(c^*)-\lambda_k d_c
\end{pmatrix}.
\end{equation}
Here $\mathbf{b}$ is determined by the linearization of $-\chi_S\nabla\cdot(S\nabla c)$ and $-\chi_R\nabla\cdot(R\nabla c)$
at the constant state $(S^*,R^*,c^*)$; its precise form is immaterial for the spectrum.

Since $\mathcal{M}(\lambda_k)$ is block upper-triangular, its eigenvalues consist of:
(i) the eigenvalues of $\mathcal{J}_F-\lambda_k D_{SR}$ and
(ii) the scalar eigenvalue $\mathcal{Q}'(c^*)-\lambda_k d_c$.
By $\mathcal{Q}'(c^*)<0$ and $d_c>0$, the scalar eigenvalue is strictly negative for all $\lambda_k\ge0$.
By the hypothesis \eqref{eq:RD_block_stable_assumption}, the eigenvalues of $\mathcal{J}_F-\lambda_k D_{SR}$ satisfy
$\Re z<0$ for every $k\ge2$.
Therefore every nonconstant mode is exponentially damped, and no diffusion/chemotaxis-driven instability occurs in the one-way case.
\end{proof}

\subsubsection*{Proof of Theorem~\ref{thm:twoway_criteria} (two-way feedback criteria)}
\begin{proof}
Consider the reduced two-way closure \eqref{eq:twoway_reduced_g} and let $\mathbf{u}^*=(S^*,R^*)$ be a homogeneous interior equilibrium.
Linearizing at $\mathbf{u}^*$ yields
\[
\partial_t \tilde{\mathbf{u}}=(\mathcal{D}-\mathcal{H})\Delta \tilde{\mathbf{u}}+\mathcal{J}_F\tilde{\mathbf{u}},
\qquad
\tilde{\mathbf{u}}=(S-S^*,R-R^*)^{\mathsf T},
\]
where $\mathcal{D}=\mathrm{diag}(d_S,d_R)$, $\mathcal{J}_F=DF(\mathbf{u}^*)$, and $\mathcal{H}$ is given in Theorem~\ref{thm:twoway_criteria}.

Projecting onto the Neumann eigenmode $\omega_k$ with $\lambda_k>0$ gives the modal ODE
\[
\frac{d}{dt}\hat{\mathbf{u}}=\mathcal{A}(\lambda_k)\hat{\mathbf{u}},
\qquad
\mathcal{A}(\lambda_k)=\mathcal{J}_F-\lambda_k(\mathcal{D}-\mathcal{H}).
\]
For each fixed $k\ge2$, $\mathcal{A}(\lambda_k)$ is a real $2\times2$ matrix.
By the Routh--Hurwitz criterion for $2\times2$ systems, the mode $k$ is unstable if and only if
\begin{equation}
\label{eq:RH_2x2}
\operatorname{tr}\mathcal{A}(\lambda_k)>0
\qquad \text{or}\qquad
\det\mathcal{A}(\lambda_k)<0,
\end{equation}
with equalities corresponding to neutral stability boundaries.

A direct computation gives the trace formula \eqref{eq:twoway_trace_condition}, proving the trace criterion.
For the determinant, expanding $\det(\mathcal{A}(\lambda))$ as a polynomial in $\lambda$ yields \eqref{eq:det_quadratic_def}
with coefficients \eqref{eq:A1A2_def}. Since $\det(\mathcal{J}_F)>0$ by \eqref{eq:kinetic_stability_twoway}, the constant term is positive.
Thus $\det(\mathcal{A}(\lambda))<0$ for some $\lambda>0$ occurs exactly in the three subcases listed in Theorem~\ref{thm:twoway_criteria}:
\begin{enumerate}[label=(\roman*)]
\item If $A_1<0$, then $\det(\mathcal{A}(\lambda))\to -\infty$ as $\lambda\to\infty$, hence some large mode destabilizes.
\item If $A_1=0$, then $\det(\mathcal{A}(\lambda))=A_2\lambda+\det(\mathcal{J}_F)$, which is negative for some $\lambda>0$ iff $A_2<0$.
\item If $A_1>0$, then the upward-opening quadratic is negative for some $\lambda>0$ iff it has two real roots and both are positive.
Since the product of the roots is $\det(\mathcal{J}_F)/A_1>0$, the roots (when real) have the same sign; thus both are positive iff
their sum $-A_2/A_1$ is positive, i.e.\ $A_2<0$, and the discriminant is positive, i.e.\ $A_2^2-4A_1\det(\mathcal{J}_F)>0$.
\end{enumerate}
Combining these determinant facts with \eqref{eq:RH_2x2} completes the proof.
\end{proof}

\subsubsection*{Proof of Corollary~\ref{cor:twoway_simple_regimes}}
\begin{proof}
Each item is an immediate sufficient condition for the determinant criterion in Theorem~\ref{thm:twoway_criteria}.
Specifically, (S1)--(S2) imply $A_1<0$, hence $\det(\mathcal{A}(\lambda))\to -\infty$ as $\lambda\to\infty$ and some large mode destabilizes.
Item (S3) is exactly the affine case $A_1=0$ with $A_2<0$.
Item (S4) is precisely the upward-opening case $A_1>0$ with $A_2<0$ and positive discriminant, ensuring negativity between the two positive roots.
\end{proof}

\begin{remark}[How the mode-wise arguments extend to the PDE signal model]
For the full PDE signal system \eqref{eq:chemo_twoway_template}, projecting onto $\omega_k$ yields a $(2+1)\times(2+1)$ linear system for
$(\hat s,\hat r,\hat c)$ in which $\hat c$ can be eliminated explicitly, producing an effective $2\times2$ matrix for $(\hat s,\hat r)$ whose entries
depend on $\lambda_k$ through the resolvent of $d_c\lambda_k-q$. This provides a direct link between the reduced closure
$c\approx g(S,R)$ and the full PDE formulation; details are placed in Appendix~\ref{app:closure}.
\end{remark}

\section{Discussion}
\label{sec:discussion}

We developed a hybrid PDE--ODE framework for two competing populations coupled to a non-motile microenvironmental switch and a
diffusive inhibitory signal, and we quantified how chemotaxis-driven transport depends on coupling directionality once the
chemotactic cue is modeled as a diffusive field $c(x,t)$.
Our analysis is intentionally mode-wise and linear: throughout, ``pattern formation'' refers to linear instability of a spatially
homogeneous equilibrium under homogeneous Neumann boundary conditions, and we distinguish this from Keller--Segel blow-up or aggregation
phenomena, which require different techniques and are not pursued here.

\subsection{Mechanistic interpretation of the directionality principle}

The core mathematical mechanism can be read directly from the modal linearizations.

\subsubsection{One-way damped signaling suppresses patterns.}
In the one-way regime (Section~\ref{subsec:main_directionality}), the chemoattractant $c$ is generated independently of $(S,R)$ and is
linearly damped (either autonomously or through a microenvironment-driven source such as $A$ with clearance).
After projection onto Neumann modes, the linearized mode matrix is block upper-triangular, so its spectrum is the union of the $(S,R)$ diffusion block and the scalar $c$-block.
Damping ensures $\mathcal{Q}'(c^*)-\lambda_k d_c<0$ for every $k\ge2$, while the $(S,R)$ block retains the same eigenvalues as the base
reaction--diffusion subsystem. Hence chemotaxis does not create new unstable eigenvalues: in this directionality class, the signal
decays without receiving feedback, and the population spectrum is unaffected. Biologically, this corresponds to cells moving along a cue
that is exogenous or microenvironmentally prescribed and cleared; such a cue can bias motion but, in the linear regime, it cannot close a
positive feedback loop that amplifies spatial heterogeneity.

\subsubsection{Two-way feedback can recover instability via an effective mobility rewrite.}
In the two-way regime, the signal depends on $(S,R)$ (directly or through fast relaxation), and the linearization generates effective
cross-diffusion. At the mode level, diffusion is no longer governed by $\mathcal{D}$ alone but by $\mathcal{D}-\mathcal{H}$, where
$\mathcal{H}$ is a feedback-induced matrix determined by chemotactic sensitivities and derivatives of the closure $c\approx g(S,R)$.
This modifies both the trace and determinant of the mode matrix $\mathcal{A}(\lambda_k)$ and allows destabilization of some nonconstant
mode even when the kinetics are stable. The explicit criteria in Theorem~\ref{thm:twoway_criteria} quantify precisely when feedback
dominates diffusion through either a trace sign change \eqref{eq:twoway_trace_condition} or determinant negativity driven by the quadratic
coefficients $A_1,A_2$ \eqref{eq:A1A2_def}. Biologically, this corresponds to cue production or activation that is locally amplified by
the populations (directly or indirectly), closing a positive loop that can convert directed migration into emergent heterogeneity.

\subsection{Limitations and modeling assumptions}

Our two-way instability criteria are stated for the reduced closure $c\approx g(S,R)$.
This is a natural approximation when the chemoattractant equilibrates on a faster time scale than cell movement, but it is not universal.
For the full PDE signal model $\partial_t c=d_c\Delta c+q\,c+h(S,R)$, eliminating $c$ yields an effective $2\times2$ mode matrix whose
coefficients depend on $\lambda_k$ through the resolvent $(d_c\lambda_k-q)^{-1}$, so the reduced ``local'' closure is best interpreted as
an asymptotic or short-scale approximation. A systematic comparison of reduced and full PDE criteria is therefore needed when the signal
time scale is comparable to population dynamics.

Our analysis is deterministic and focuses on mode-wise linear stability under Neumann boundary conditions.
In applications, however, intrinsic fluctuations stemming from finite population sizes and stochastic phenotype flips may be non-negligible.
Complementary stochastic differential equation-based analysis frameworks for 
pure demographic noise establish global well-posedness,
and show that demographic noise alone does not induce population extinction; see \citep{abundo_stochastic_1991,wang_analysis_2025-1}.
In the present hybrid PDE--ODE context, incorporating demographic noise (or hybrid jump perturbations of the local
switching kinetics) is a natural next step, and a key question is how such intrinsic noise interacts with diffusion-driven
mode selection and feedback-induced effective cross-diffusion.

The hybrid PDE--ODE structure introduces parameters governing switching (e.g.\ $\theta,\beta$), clearance/decay (e.g.\ $\gamma_d$),
and chemotactic sensitivity (e.g.\ $\chi_S,\chi_R$), in addition to baseline growth and conversion rates.
Identifying these parameters from experimental data remains nontrivial, especially when only partial observations of microenvironmental
states are available. Practical calibration will likely require combining time-resolved spatial data with reduced-order summaries (e.g.\
dominant Fourier modes) informed by the explicit mode-wise criteria.

\subsection{Outlook and extensions}

Several directions follow naturally from the present analysis.

A first extension is to derive instability criteria directly for the PDE--PDE--PDE system with $c$ governed by
$\partial_t c=d_c\Delta c+q\,c+h(S,R)$, using mode-wise elimination of $c$ (Schur complements) to obtain sharp conditions in terms of
$(d_c,q)$ and the Laplacian spectrum. This would clarify precisely when the reduced closure $c\approx g(S,R)$ is accurate and how nonlocal
signal effects shift thresholds.

While linear instability is an entry point for pattern formation, the observed long-time patterns depend on nonlinear saturation,
possible amplitude equations near onset, and secondary instabilities. Extending the analysis to weakly nonlinear regimes (e.g.\ via
center-manifold reductions or normal forms) would connect the trace/determinant thresholds to pattern selection and robustness.

Our results hold for $N\in\{1,2,3\}$ at the level of mode-wise criteria, but geometry enters through the Neumann spectrum and can strongly
influence which modes are realized. Systematic exploration on three-dimensional domains with realistic boundary geometry is therefore
important for tumor microenvironment applications, where confinement, interfaces, and anisotropy can bias pattern selection.

The present microenvironment is a two-state switch; richer switching networks (multiple phenotypes, irreversible transitions, or
history-dependent states) \citep{jain_dynamical_2023,friedman_hysteresis_2014} can be incorporated while preserving the hybrid PDE--ODE structure. On the transport side, nonlinear
sensitivities and density-dependent motility \citep{wang_chemotaxis_2010,hofer_dictyostelium_1995,kowalczyk_preventing_2005,choi_prevention_2010,lushnikov_macroscopic_2008} can be introduced, with the expectation that directionality (presence/absence of feedback)
will remain the organizing principle, while the explicit thresholds will be modified.

In summary, the analysis identifies directionality as a structural determinant of chemotaxis-driven patterning in hybrid PDE--ODE
models: without feedback, damping in the diffusive cue prevents spectral destabilization, whereas with feedback the effective mobility is
rewritten and explicit mode-wise criteria predict when diffusion/chemotaxis-driven instabilities can emerge. This provides a tractable
bridge between reaction--diffusion and chemotaxis paradigms in a microenvironmentally structured setting.

\appendix

\section{Spectral preliminaries}
\label{app:spectral}

This appendix collects standard spectral facts for the Neumann Laplacian that are used in the main text, in particular for
(i) mode-wise stability arguments under homogeneous Neumann boundary conditions, and
(ii) the spectral representation and smoothing of the damped heat equation for $D$.
For completeness we record statements in the form needed for reproducibility; detailed proofs can be found in standard references.

\subsection{Neumann eigenpairs and basic properties}
\label{app:neumann_eigs}

Let $U\subset\mathbb{R}^N$ ($N\in\{1,2,3\}$) be a bounded domain with smooth boundary $\partial U$.
Consider the Neumann eigenvalue problem
\begin{equation}
\label{eq:app_neumann_problem}
\begin{cases}
-\Delta \omega_k=\lambda_k \omega_k & \text{in }U,\\[2pt]
\partial_{\mathrm n}\omega_k=0 & \text{on }\partial U,
\end{cases}
\end{equation}
where $\partial_{\mathrm n}$ denotes the outward normal derivative.
Classical spectral theory \citep{evans_partial_2008} yields a complete orthonormal basis $\{\omega_k\}_{k\ge1}$ of $L^2(U)$ consisting of smooth eigenfunctions
and a nondecreasing sequence of eigenvalues
\begin{equation}
\label{eq:app_neumann_order}
0=\lambda_1<\lambda_2\le\lambda_3\le\cdots,\qquad \lambda_k\to\infty\ \text{as }k\to\infty.
\end{equation}
The first eigenfunction $\omega_1$ is constant (up to normalization), and modes $k\ge2$ correspond to nonconstant spatial perturbations.

\subsection{Weyl-type eigenvalue bounds}
\label{app:weyl_bounds}

We use a quantitative form of Weyl's law to control summability of eigenfunction expansions.

\begin{lemma}[Weyl-type bounds]\citep{netrusov_weyl_2005}
\label{lem:weyl_bounds}
There exist constants $C_1,C_2>0$ and an integer $k_0\ge2$ such that
\begin{equation}
\label{eq:weyl_bounds}
C_1\,k^{2/N}\le \lambda_k \le C_2\,k^{2/N},
\qquad \forall k\ge k_0.
\end{equation}
\end{lemma}

\begin{proof}
This follows from Weyl asymptotics for the Neumann Laplacian on smooth bounded domains.
\end{proof}

\subsection{Sup-norm derivative bounds for eigenfunctions}
\label{app:eig_supnorm}

The following estimate is a convenient sufficient condition for termwise differentiation of eigenfunction expansions.

\begin{lemma}[Sup-norm derivative bound]
\label{lem:eig_derivative_supnorm}
Let $(\lambda_k,\omega_k)$ satisfy \eqref{eq:app_neumann_problem} and assume $\|\omega_k\|_{L^2(U)}=1$.
For every multi-index $\alpha\in\mathbb{N}^N$ there exists $C=C(\alpha,U)>0$ such that
\begin{equation}
\label{eq:eig_sup_derivative}
\|\partial^\alpha \omega_k\|_{L^\infty(U)}\le C(\alpha,U)\,\lambda_k^{|\alpha|/2+1},
\qquad k\ge1,
\end{equation}
where $|\alpha|=\alpha_1+\cdots+\alpha_N$.
Consequently, for each integer $m\ge1$ there exists $C=C(m,U)>0$ such that
\begin{equation}
\label{eq:eig_Hm_bound}
\|\omega_k\|_{H^m(U)}\le C(m,U)\,(1+\lambda_k)^{m/2+1}.
\end{equation}
\end{lemma}

\begin{proof}
The estimate follows from elliptic regularity applied iteratively to $-\Delta\omega_k=\lambda_k\omega_k$ together with Sobolev embeddings
$H^{m}(U)\hookrightarrow C^{m-2}(\overline U)$ for $N\le3$.
\end{proof}

\subsection{Spectral representation for the damped heat equation}
\label{app:damped_heat_spectral}

For completeness we state the spectral formula used for the damped signal $D$.
Consider
\begin{equation}
\label{eq:app_damped_heat}
\partial_t D = d_D\Delta D-\gamma_d D
\quad\text{in }U\times(0,\infty),\qquad
\partial_{\mathrm n}D=0\ \text{on }\partial U\times(0,\infty),\qquad
D(\cdot,0)=D_0.
\end{equation}

\begin{proposition}[Spectral solution formula and smoothing]
\label{prop:damped_heat_spectral}
Assume $D_0\in L^2(U)$ and let $(\lambda_k,\omega_k)$ be the Neumann eigenpairs from \eqref{eq:app_neumann_problem}.
Define $c_k\coloneqq (D_0,\omega_k)_{L^2(U)}$.
Then the unique solution of \eqref{eq:app_damped_heat} admits the spectral representation
\begin{equation}
\label{eq:damped_heat_series}
D(x,t)=\sum_{k=1}^\infty c_k\,e^{-(d_D\lambda_k+\gamma_d)t}\,\omega_k(x),
\qquad t\ge0,
\end{equation}
with convergence in $L^2(U)$ uniformly for $t\ge0$.
Moreover, for every $\varepsilon>0$ and every multi-index $\alpha$, the differentiated series
\[
\partial^\alpha D(x,t)=\sum_{k=1}^\infty c_k\,e^{-(d_D\lambda_k+\gamma_d)t}\,\partial^\alpha\omega_k(x)
\]
converges uniformly on $\overline U\times[\varepsilon,\infty)$, hence $D(\cdot,t)\in C^\infty(\overline U)$ for all $t>0$.
\end{proposition}

\begin{proof}
The expansion \eqref{eq:damped_heat_series} follows by separation of variables in the $\{\omega_k\}$ basis \citep{brezis_functional_2011}.
Uniform convergence of derivatives on $[\varepsilon,\infty)$ follows from Lemma~\ref{lem:eig_derivative_supnorm} and Weyl bounds
(Lemma~\ref{lem:weyl_bounds}), since $e^{-d_D\lambda_k\varepsilon}$ dominates any polynomial growth in $\lambda_k$.
\end{proof}

\section{From the diffusive signal equation to an effective closure}
\label{app:closure}

This appendix formalizes how a diffusive chemoattractant equation can be eliminated mode-by-mode, yielding a reduced
closure that modifies the effective mobility of $(S,R)$ and leads to the cross-diffusion structure used in the
mode-wise instability criteria.

\subsection{A linearized two-way PDE--PDE system around a homogeneous equilibrium}
\label{app:closure_linearization}

Consider the two-way chemotaxis extension with a diffusive signal
\begin{align}
\label{eq:app_twoway_c_pde}
\begin{cases}
\partial_t S = d_S\Delta S - \chi_S\nabla\!\cdot(S\nabla c) + f_S(S,R),\\[2pt]
\partial_t R = d_R\Delta R - \chi_R\nabla\!\cdot(R\nabla c) + f_R(S,R),\\[2pt]
\partial_t c = d_c\Delta c + q\,c + h(S,R),\qquad q<0,
\end{cases}
\end{align}
posed on $U$ with homogeneous Neumann boundary conditions for $S,R,c$.
Let $(S^*,R^*,c^*)$ be a spatially homogeneous equilibrium satisfying
\[
f_S(S^*,R^*)=0,\qquad f_R(S^*,R^*)=0,\qquad q\,c^*+h(S^*,R^*)=0.
\]
Write perturbations $\tilde S=S-S^*$, $\tilde R=R-R^*$, $\tilde c=c-c^*$ and denote the Jacobians
\[
\mathcal{J}_F \coloneqq DF(S^*,R^*)=
\begin{pmatrix}
a & b\\ c & d
\end{pmatrix},
\qquad
\mathcal{J}_h \coloneqq \nabla h(S^*,R^*)=(h_S^*,h_R^*).
\]
Since $\nabla c^*=0$, linearization of the chemotaxis terms yields
\[
-\chi_S\nabla\cdot(S\nabla c) \approx -\chi_S\nabla\cdot(S^*\nabla \tilde c)= -\chi_S S^*\Delta \tilde c,
\qquad
-\chi_R\nabla\cdot(R\nabla c) \approx -\chi_R R^*\Delta \tilde c.
\]
Hence the linearized system is
\begin{align}
\label{eq:app_linearized_system}
\begin{cases}
\partial_t \tilde S = d_S\Delta \tilde S - \chi_S S^*\Delta \tilde c + a\,\tilde S + b\,\tilde R,\\[2pt]
\partial_t \tilde R = d_R\Delta \tilde R - \chi_R R^*\Delta \tilde c + c\,\tilde S + d\,\tilde R,\\[2pt]
\partial_t \tilde c = d_c\Delta \tilde c + q\,\tilde c + h_S^*\tilde S + h_R^*\tilde R.
\end{cases}
\end{align}

\subsection{Mode-wise decomposition and Schur complement elimination}
\label{app:closure_schur}

Let $\{(\lambda_k,\omega_k)\}_{k\ge1}$ be the Neumann Laplacian eigenpairs on $U$:
\[
-\Delta \omega_k=\lambda_k\omega_k,\qquad \partial_{\mathrm n}\omega_k=0\ \text{on }\partial U,\qquad
0=\lambda_1<\lambda_2\le\cdots,\quad \lambda_k\to\infty.
\]
Project \eqref{eq:app_linearized_system} onto a fixed mode $k\ge2$ and write
\[
\tilde S(x,t)=\hat s_k(t)\,\omega_k(x),\quad
\tilde R(x,t)=\hat r_k(t)\,\omega_k(x),\quad
\tilde c(x,t)=\hat c_k(t)\,\omega_k(x).
\]
Using $\Delta \omega_k=-\lambda_k\omega_k$, we obtain the $3\times3$ linear ODE system
\begin{equation}
\label{eq:app_mode_3by3}
\frac{d}{dt}
\begin{pmatrix}\hat s_k\\[4pt]
\hat r_k\\[4pt]
\hat c_k\end{pmatrix}
=
\underbrace{\begin{pmatrix}
a-d_S\lambda_k & b & \chi_S S^*\lambda_k\\[4pt]
c & d-d_R\lambda_k & \chi_R R^*\lambda_k\\[4pt]
h_S^* & h_R^* & q-d_c\lambda_k
\end{pmatrix}}_{\eqqcolon\,\mathcal{M}(\lambda_k)}
\begin{pmatrix}\hat s_k\\[4pt]
\hat r_k\\[4pt]
\hat c_k\end{pmatrix}.
\end{equation}
For spectral analysis, one may eliminate $\hat c_k$ by a Schur complement argument.
Let $\mu\in\mathbb{C}$ be a candidate growth rate for mode $k$ and consider the eigenvalue problem
\[
\mu \mathbf{z}=\mathcal{M}(\lambda_k)\mathbf{z},\qquad \mathbf{z}=(\hat s_k,\hat r_k,\hat c_k)^{\mathsf T}.
\]
Assuming $\mu\neq q-d_c\lambda_k$ (the signal eigenvalue of the uncoupled $c$-block), the third row gives
\begin{equation}
\label{eq:app_ck_eliminate}
(\mu-(q-d_c\lambda_k))\,\hat c_k = h_S^*\,\hat s_k + h_R^*\,\hat r_k,
\qquad\Longrightarrow\qquad
\hat c_k = \frac{h_S^*\,\hat s_k + h_R^*\,\hat r_k}{\mu-(q-d_c\lambda_k)}.
\end{equation}
Substituting \eqref{eq:app_ck_eliminate} into the first two rows yields an equivalent $2\times2$ eigenvalue problem
\begin{equation}
\label{eq:app_schur_effective_2by2}
\mu
\begin{pmatrix}\hat s_k\\ \hat r_k\end{pmatrix}
=
\mathcal{A}_{\mathrm{eff}}(\mu;\lambda_k)
\begin{pmatrix}\hat s_k\\ \hat r_k\end{pmatrix},
\end{equation}
where the effective mode matrix is
\begin{equation}
\label{eq:app_Aeff_mu}
\mathcal{A}_{\mathrm{eff}}(\mu;\lambda_k)
\coloneqq
\Bigl(\mathcal{J}_F-\lambda_k \mathcal{D}\Bigr)
+
\frac{\lambda_k}{\mu-(q-d_c\lambda_k)}\,
\begin{pmatrix}
\chi_S S^*\\ \chi_R R^*
\end{pmatrix}
\begin{pmatrix}
h_S^* & h_R^*
\end{pmatrix},
\qquad
\mathcal{D}\coloneqq \mathrm{diag}(d_S,d_R).
\end{equation}
Equivalently, the characteristic equation of $\mathcal{M}(\lambda_k)$ factors as
\begin{equation}
\label{eq:app_char_factor}
\det(\mu I-\mathcal{M}(\lambda_k))
=
\bigl(\mu-(q-d_c\lambda_k)\bigr)\,
\det\!\Bigl(\mu I-\mathcal{A}_{\mathrm{eff}}(\mu;\lambda_k)\Bigr).
\end{equation}
Thus, for each fixed $\lambda_k>0$, spectral stability reduces to understanding the nonlinear (in $\mu$) $2\times2$ determinant
in \eqref{eq:app_char_factor}. The next subsection explains when \eqref{eq:app_Aeff_mu} simplifies to the reduced closure
used in the main text.

\subsection{Fast-relaxation regime and reduced closure}
\label{app:closure_fast_relaxation}

A standard fast-relaxation scaling introduces $\varepsilon\ll 1$ in the signal equation:
\begin{equation}
\label{eq:app_fast_relax}
\varepsilon\,\partial_t c = d_c\Delta c + q\,c + h(S,R),\qquad q<0.
\end{equation}
Linearizing \eqref{eq:app_fast_relax} replaces the third diagonal entry $q-d_c\lambda_k$ in \eqref{eq:app_mode_3by3} by
$\varepsilon^{-1}(q-d_c\lambda_k)$, hence \eqref{eq:app_ck_eliminate} becomes
\begin{equation}
\label{eq:app_ck_fast}
\hat c_k = \frac{h_S^*\,\hat s_k + h_R^*\,\hat r_k}{\varepsilon\mu-(q-d_c\lambda_k)}.
\end{equation}
If $\varepsilon|\mu|\ll |q-d_c\lambda_k|$ (signal relaxes faster than the growth/decay rate of the $(S,R)$-mode),
then
\begin{equation}
\label{eq:app_ck_qss}
\hat c_k \;=\; -\frac{h_S^*\,\hat s_k + h_R^*\,\hat r_k}{q-d_c\lambda_k}\;+\;\mathcal{O}(\varepsilon),
\qquad \varepsilon\to0,
\end{equation}
and the effective matrix \eqref{eq:app_Aeff_mu} admits the approximation
\begin{equation}
\label{eq:app_Aeff_qss}
\mathcal{A}_{\mathrm{eff}}(\mu;\lambda_k)
=
\Bigl(\mathcal{J}_F-\lambda_k \mathcal{D}\Bigr)
-\frac{\lambda_k}{q-d_c\lambda_k}\,
\begin{pmatrix}
\chi_S S^*\\ \chi_R R^*
\end{pmatrix}
\begin{pmatrix}
h_S^* & h_R^*
\end{pmatrix}
\;+\;\mathcal{O}(\varepsilon).
\end{equation}
In particular, for each fixed mode $\lambda_k>0$, the quasi-steady elimination yields a linear $2\times2$ mode matrix
(with no dependence on $\mu$ at leading order). This provides a rigorous modal interpretation of the reduced closure.

\subsubsection{Connection to a local closure $c\approx g(S,R)$.}
The quasi-steady relation in physical space reads
\begin{equation}
\label{eq:app_qss_operator}
0 = d_c\Delta c + q\,c + h(S,R)
\qquad\Longrightarrow\qquad
c = -\,(d_c\Delta+q)^{-1} h(S,R),
\end{equation}
with Neumann boundary conditions. This is a nonlocal closure in general.
However, in regimes where either
(i) $d_c$ is small compared to $|q|$ (reaction-dominated signal), or
(ii) one focuses on sufficiently short spatial scales for which $(q-d_c\lambda_k)^{-1}$ is approximately constant across the
unstable band, or
(iii) one adopts a model reduction retaining only a finite set of dominant modes,
\eqref{eq:app_qss_operator} can be approximated by a local closure $c\approx g(S,R)$ near equilibrium, namely
\begin{equation}
\label{eq:app_local_g}
c \approx g(S,R)
\quad\text{with}\quad
\nabla g(S^*,R^*) = -\frac{1}{q}\,\nabla h(S^*,R^*)\;=\;-\frac{1}{q}\,(h_S^*,h_R^*).
\end{equation}
Under \eqref{eq:app_local_g}, linearization of the reduced chemotaxis term
$-\chi\nabla\cdot(S\nabla g(S,R))$ produces precisely the rank-one correction to the diffusion matrix encoded by the outer product
\[
\begin{pmatrix}\chi_S S^*\\ \chi_R R^*\end{pmatrix}
\begin{pmatrix}g_S^* & g_R^*\end{pmatrix},
\qquad
(g_S^*,g_R^*)=\nabla g(S^*,R^*),
\]
which is the effective cross-diffusion structure used in Theorem~\ref{thm:twoway_criteria}.

\begin{remark}[When the reduced cross-diffusion form is accurate]
The reduction to a local closure $c\approx g(S,R)$ is most accurate when the signal equation is strongly damped ($|q|$ large)
and the feedback $h(S,R)$ varies on spatial scales not significantly affected by the resolvent $(d_c\Delta+q)^{-1}$, or when the
analysis targets an instability band where $(q-d_c\lambda_k)^{-1}$ may be treated as approximately constant.
If one wishes to avoid any locality assumption, the mode-wise elimination \eqref{eq:app_char_factor}--\eqref{eq:app_Aeff_qss}
already provides a fully rigorous route: for each $\lambda_k$ the $(S,R)$-mode matrix is modified by the scalar factor
$-(q-d_c\lambda_k)^{-1}$ multiplying the rank-one feedback operator.
\end{remark}

\section*{Data Availability:}
Data sharing is not applicable to this article, as no new data was created or analyzed in this study.

\section*{Author Contribution:}

All authors have accepted responsibility for the entire content of this manuscript and approved its submission.


\section*{Declaration:}
All authors declare no competing interests.

\bibliography{reference}

\end{document}